\documentclass[12pt]{amsart}
\usepackage{latexsym}
\usepackage{amssymb, amsmath,mathrsfs, mathtools}
\usepackage[left=2.5cm,right=2.5cm,top=2.5cm, bottom=2.5cm ]{geometry}
\usepackage[OT2,T1]{fontenc}

\usepackage{setspace}
\usepackage{graphicx, subfigure}

\usepackage[hypertexnames=false, colorlinks, citecolor=red, linkcolor=red]{hyperref}
\makeatletter
\@namedef{subjclassname@2020}{%
  \textup{2020} Mathematics Subject Classification}
\makeatother

\newtheorem{theorem}{Theorem}[section]
\newtheorem{lemma}[theorem]{Lemma}

\newtheorem{conjecture}[theorem]{Conjecture}
\newtheorem{example}[theorem]{Example}
\newtheorem{question}{Question}
\newtheorem{corollary}[theorem]{Corollary}
\newtheorem{proposition}[theorem]{Proposition}

\theoremstyle{remark}
\newtheorem{remark}[theorem]{Remark}

\newcommand{\RE}{\mathrm{Re}\,}
\newcommand{\IM}{\mathrm{Im}\,}

\DeclareSymbolFont{cyrletters}{OT2}{wncyr}{m}{n}
\DeclareMathSymbol{\Sha}{\mathalpha}{cyrletters}{"58}
\author[Bickel]{Kelly Bickel$^\dagger$}
\address{Department of Mathematics, Bucknell University, 360 Olin Science Building, Lewisburg, PA 17837, USA.}
\email{kelly.bickel@bucknell.edu}
\thanks{$\dagger$ Research supported in part by National Science Foundation
DMS grant \#2000088.}


\author[Gorkin]{Pamela Gorkin}
\address{Department of Mathematics, Bucknell University, 360 Olin Science Building, Lewisburg, PA 17837, USA.}
\email{pgorkin@bucknell.edu}

\begin{document}
\raggedbottom
\title{Blaschke Products, Level Sets, and Crouzeix's Conjecture}
\date{\today}

\maketitle
\begin{abstract}
We study several problems motivated by Crouzeix's conjecture, which we consider in the special setting of model spaces and compressions of the shift with finite Blaschke products as symbols. We pose a version of the conjecture in this setting, called the level set Crouzeix (LSC) conjecture, and establish structural and uniqueness properties for (open) level sets of finite Blaschke products that allow us to prove the LSC conjecture in several cases. In particular, we use the geometry of the numerical range to prove the LSC conjecture for compressions of the shift corresponding to unicritical Blaschke products of degree $4$.
\end{abstract}
\section{Introduction} Let $A$ be an $n\times n$ matrix and let $W(A)$ denote its numerical range
\[ W(A) =\left \{ \left \langle A x, x \right \rangle: x\in \mathbb{C}^n, \| x\|=1\right \},\]
 an important subset in the plane that both contains the spectrum of $A$ and encodes additional properties of $A$. A famous open problem 
 related to numerical ranges is Crouzeix's conjecture from  \cite{Cr07}, which states
 \begin{conjecture}[Crouzeix's Conjecture] If $p$ is a polynomial, then $\displaystyle \| p(A) \| \le 2 \max_{z\in W(A)} |p(z)|.$
 \end{conjecture}
Numerical evidence for the conjecture can be found in \cite{GO} and applications thereof appear in \cite{CG}. In \cite{CP17}, Crouzeix and Palencia showed that the conjecture is true if $2$ is replaced by $1+ \sqrt{2}$. Both Crouzeix and a variety of other researchers have established Crouzeix's conjecture in a number of special cases (see \cite{BCD, CH13, Crouzeix1, C13, GKL, GreenbaumChoi} and the survey paper \cite{Badea}), but the problem remains open even for $3\times3$ matrices. 

Two cases, \cite{GreenbaumChoi} and \cite{C13}, motivated the study here. In the first of these papers, the authors studied perturbed Jordan blocks, or $(n+1) \times (n+1)$ matrices of the form
\begin{equation}
J_{n + 1, a} := \begin{pmatrix} 0 & 1 & \\ & \ddots & \ddots \\ & & \ddots & 1\\ a & & & 0 \end{pmatrix},\label{perturbed}\end{equation} with $a \in \mathbb{C}$. For $|a| < 1$, one can check that these matrices are contractions for which all eigenvalues lie inside the open unit disk $\mathbb{D}$ and the $\mbox{rank}~(I - A^\ast A) = \mbox{rank}~( I - A A^\ast) = 1$. Such matrices, which include Jordan blocks, represent operators called compressions of the shift operator. Following this line of study, we were naturally led to a certain class of nilpotent operators. The study of such $3 \times 3$ matrices was the subject of the work in \cite{C13}. In this paper, we consider various situations in which compressed shift operators satisfy Crouzeix's conjecture. We turn now to a discussion of such operators. 

Let $\Theta$ denote a degree-$n$ finite Blaschke product, i.e. 
\[ \Theta(z) =\lambda \prod_{j=1}^n \frac{z-a_j}{1-\bar{a}_j z}, \quad \text{ for } a_1, \dots, a_n \in \mathbb{D} \text{ and } \lambda \in \mathbb{T}.\]
Finite Blaschke products comprise a special class of inner functions, that is, functions that are bounded and holomorphic on the unit disk $\mathbb{D}$ with radial boundary values on the unit circle $\mathbb{T}$ that have modulus one a.e. For $\Theta$ a finite Blaschke product, let $H^2=H^2(\mathbb{D})$ denote the standard Hardy space on the unit disk $\mathbb{D}$ and let $K_\Theta = H^2 \ominus \Theta H^2$ denote the model space associated to $\Theta$. If $M_z$ denotes multiplication by the independent variable $z$, then we can define the associated compression of the shift $S_\Theta$ by 
\[ S_\Theta = P_\Theta M_z |_{K_\Theta},\]
where $P_{\Theta}$ is the orthogonal projection from $H^2$ onto $K_\Theta$. Since $\deg B = n$, the space $K_\Theta$ has dimension $n$ and so, we can interpret $S_\Theta$ as an $n\times n$ matrix by writing down its representation with respect to an orthonormal basis of $K_\Theta$. These operators $S_\Theta$ are particularly important because Sz.-Nagy and Foias \cite{SFBK10} showed that every completely non-unitary $n\times n$ contraction of class $C_0$ with defect index $1$ is unitarily equivalent to $S_\Theta$ for some finite Blaschke product $\Theta$.  

In this paper, we are interested in Crouzeix's conjecture for such compressions of the shift. This is a natural class to study for two reasons. First, it includes several classes of matrices, for example, Jordan blocks and perturbed Jordan blocks, for which Crouzeix's conjecture is known; these correspond to the finite Blaschke products $\Theta = z^n$ and $\Theta = \frac{z^n -a}{1-\bar{a}z^n}$. Second, the numerical ranges of compressions of the shift satisfy particularly nice geometric properties including one called the Poncelet property, see for example \cite{GauWu, M98}.

In this paper, we propose a conjecture related to Crouzeix's conjecture that is specific to the behavior of compressions of the shift and finite Blaschke products. To state it, observe that in \cite{Sar67}, Sarason proved that if $f$ is bounded and holomorphic on $\mathbb{D},$ then 
\[ f(S_\Theta) = P_{\Theta} M_f |_{K_{\Theta}},\]
where $M_f$ is multiplication by $f$. 
Furthermore, his work appears to imply that if $B$ is a finite Blaschke product with $\deg B < \deg \Theta$, then $\| B(S_\Theta) \| =1$. Garcia and Ross provide more details and a direct statement of this result in Corollary $4$ in \cite[p. 512]{GarciaRoss}).  Combining these facts with Crouzeix's conjecture yields the following new complex analysis conjecture:

\begin{conjecture}[Level set Crouzeix conjecture] Let $\Theta, B$ be finite Blaschke products with $\deg B < \deg \Theta$. Then 
\begin{equation} \label{eqn:LSC}  \max \left \{ |B(z)|: z\in W(S_\Theta)\right \} \ge \tfrac{1}{2}. \end{equation} 
\end{conjecture}
This is a conjecture about classical holomorphic functions on $\mathbb{D}$ and proposes a non-obvious relationship between the level sets of finite Blaschke products and the numerical ranges of compressions of the shift. Indeed, for $B$ a finite Blaschke product and $r \in (0,1)$, define the open level set
\[ \Omega^B_{r} =\left \{ z\in \mathbb{C}: |B(z)| < \tfrac{1}{2} \right\}.\]
Then if $\Theta$ is a finite Blaschke product with $\deg \Theta > \deg B$, the level set Crouzeix conjecture (LSC conjecture) asserts that $W(S_\Theta)$ cannot be contained in the level set $\Omega^B_{1/2}$. In what follows, if \eqref{eqn:LSC} holds for a particular pair $(B, \Theta)$, we will say that $(B, \Theta)$ satisfies the level set Crouzeix inequality (LSC inequality). 

Level sets of inner functions have been studied in other contexts. For example, recall that an inner function $\Theta$ is a one-component inner function if there exists an $\varepsilon > 0$ such that $\Omega_\varepsilon^\Theta$ is connected. These were introduced by B. Cohn \cite{Cohn} and in the interim have been heavily studied, see \cite{baranov, CM, CM2, Nicolau}. Cohn introduced this class because he was able to characterize the Carleson measures for the model spaces $H^2 \ominus \Theta H^2$, under the assumption that $\Theta$ was a one-component inner function. One can check that all finite Blaschke products are one-component inner functions. 

In this paper, we establish results about the individual level sets of finite Blaschke products, even in the setting where they have more than one component.  

\subsection{Outline and Main Results} This paper handles three interconnected topics; it establishes structural and uniqueness properties for level sets of finite Blaschke products, proves the LSC conjecture in a number of special cases, and provides an in-depth study of the  LSC conjecture and related topics in the setting of unicritical $B$ or $\Theta$. First, because investigating the LSC conjecture requires detailed knowledge of level set behavior, Section \ref{sec:LS} establishes several results about the structure of level sets of finite Blaschke products. A key result is this statement about uniqueness:

\begin{theorem} \label{thm:B1}  Let $B$ and $C$ be finite Blaschke products with $\deg B = \deg C$. If there is some $r \in (0,1)$ with $\Omega_r^B \subseteq \Omega_r^C$, then there exists $\lambda \in \mathbb{T}$ such that  $B = \lambda C$.

\end{theorem}

This theorem appears later as Theorem \ref{thm:LSB}. It is related to level set investigations for inner functions conducted by Berman and Stephenson and Sundberg \cite{B84, Stephenson2}.  While Berman studied inner functions that share level sets for possibly different values, Stephenson and Sundberg studied inner functions whose level sets corresponding to the same value have boundaries with a subarc in common. Both results imply that if two inner functions share a common level set for some value $r$ in $(0,1)$, then the inner functions must agree up to a unimodular constant.  Theorem \ref{thm:B1} displays a similar relationship but  it restricts to finite Blaschke products of the same degree and only requires set containment, not equality.  The proof rests on Theorem \ref{HRrevised}, which generalizes a result of Horwitz and Rubel characterizing when two monic Blaschke products are equal. As the proof of  Theorem \ref{HRrevised} uses many of original arguments of Horwitz and Rubel, we postpone its proof  to Section \ref{sec:RH}.

Section \ref{sec:LSC} investigates the LSC inequality for different classes of $(B,\Theta)$ using a variety of tools. The employed techniques often involve the analysis of pseudohyperbolic disks, denoted $D_\rho(z_0,r)$, with given (pseudohyperbolic) centers $z_0 \in \mathbb{D}$ and (pseudohyperbolic) radii $r$, with $0 < r < 1$, that is,
 \[ D_\rho(z_0,r) = \left \{ z\in \mathbb{D}: \left | \frac{z-z_0}{1-\bar{z}_0 z} \right | <r \right\}.\]
The four main cases we handle are encoded in the following theorem:
\begin{theorem}\label{thm:LSC} Let $\Theta, B$ be finite Blaschke products with $\deg B < \deg \Theta$. Then the LSC inequality \eqref{eqn:LSC} holds for $(B, \Theta)$ in all of the following cases:
\begin{itemize}
\item[i.] $\deg\Theta =n$ for which there is a pseudohyperbolic disk $D_\rho(z_0, (\tfrac{1}{2})^{1/(n-1)}) \subseteq W(S_\Theta)$.
\item[ii.] $\deg B =2$ and $\deg \Theta \ge 6$.
\item[iii.] $B(0)=0$, $|B'(0)| \ge \tfrac{2\sqrt{2}}{3}$, $\deg \Theta \ge 9$, and $0 \in W(S_\Theta)$.
\item[iv.] $\deg B =2$ and $\Omega^B_{1/2}$ has two components.
\end{itemize}
\end{theorem}

These results appear later as   Corollary \ref{cor:main}, Theorem \ref{thm:26}, Theorem \ref{thm:hoffman}, and Theorem \ref{thm:degree2level}, respectively. The proofs employ a variety of tools including the uniqueness result Theorem \ref{thm:B1}, Fuss's formula for disks inscribed in quadrilaterals that are circumscribed by the unit circle $\mathbb{T}$ (see \cite{Hess}), geometric properties of the numerical ranges $W(S_\Theta)$, and various relationships between Blaschke product level sets and related pseudohyperbolic disks.

The final theme of the paper concerns a deep study of the unicritical case. Here, a finite Blaschke product $\phi$ is called unicritical if there is a $z_0 \in \mathbb{D}$, $n \ge 1$, and $\lambda \in \mathbb{T}$ such that 
\[ \phi(z) = \lambda \left( \frac{z-z_0}{1-\bar{z}_0 z}\right)^n.\]
In Section \ref{sec:uniB}, we let $B$ be unicritical and prove the following result.

\begin{theorem}  Let $\Theta$ and $B$ be finite Blaschke products with $\deg B <\deg \Theta$ and  $B$ unicritical. Then the LSC inequality \eqref{eqn:LSC} holds for $(B, \Theta)$.  \end{theorem}

This appears later as  Corollary \ref{cor:uniB}. It follows from properties of finite Blaschke products and classical results about spectral set properties of disks.

In Section \ref{sec:uniTheta}, we let $\Theta$ be unicritical. The statement of \eqref{eqn:LSC} suggests that we need to study the numerical range of the associated compression of the shift $S_\Theta$.  For unicritical $\Theta$,  these numerical ranges were studied by Gaaya in \cite{Gaaya10, Gaaya12}, Gau and Wu in \cite{GauWu13, GauWu14}, and in work of Partington and the second author \cite{GorkinPartington}. When $\deg \Theta =3$, results about $W(S_\Theta)$ are encoded in Crouzeix's work \cite{C13}; indeed, his arguments imply that the full Crouzeix conjecture holds for $S_\Theta$ when $\deg \Theta=3$ and $\Theta$ is unicritical.

We contribute to this area by identifying a natural curve $\mathcal{C}$ that lies inside $W(S_\Theta)$, see Proposition \ref{prop:Ct}. Using this curve, we are able to identify large pseudohyperbolic disks that lie inside of $W(S_\Theta)$ for $\Theta$ with small degree:

\begin{theorem} \label{thm:pshds} Let $\Theta$ be unicritical with $\deg \Theta =n$. Then:
\begin{itemize}
\item[i.] If $n \ge 3$, $W(S_{\Theta})$ always contains a pseudohyperbolic disk of radius $\tfrac{1}{2^{1/2}}$. 
\item[ii.]  If $n \ge 4$, $W(S_{\Theta})$ always contains a pseudohyperbolic disk of radius $\tfrac{1}{2^{1/3}}$.
\end{itemize}
\end{theorem}

This is encoded in Theorem \ref{prop: phd} and follows from some technical estimates showing that certain disks must be contained in the convex hull of the associated $\mathcal{C}$ curves. We conjecture that similar results hold for larger degree $\Theta$ but even when $\deg \Theta =5$, the computations become much more complicated, see Remark \ref{rem:n5}. Then Theorem \ref{thm:pshds} paired with Theorem \ref{thm:LSC}(i) shows that the LSC inequality holds for all $(B,\Theta)$ with $\Theta$ unicritical and $\deg B < \deg \Theta \le 4.$  The $n=3$ case also follows from Crouzeix's work \cite{C13}, but the $n=4$ case appears to be new. 

Section \ref{sec:uniTheta} also investigates the full Crouzeix conjecture for low degree unicritical $\Theta$, which we often denote by $\Theta_t$ to indicate the case when the repeated zero occurs at some $t\in [0,1)$. Basically, we use the associated curves $\mathcal{C}$ in $W(S_{\Theta_t})$ to show that in the $n=3,4,5$ cases, $W(S_{\Theta_t})$ is a $\| X_t \| \cdot \|X_t^{-1}\|$-spectral set for matrices $X_t$ given in \eqref{eqn:Xt3}, \eqref{eqn:Xt4}, and \eqref{eqn:Xt5} respectively. This means that for all polynomials $p$,
\[ \| p(S_\Theta) \|  \le \| X_t \| \cdot \|X_t^{-1}\| \max_{z\in W(S_\Theta)} |p(z)|.\]
We can often estimate this constant. In the $n=3$ case, one can check that for $t \in [0,1)$,
 \[ 2 \le \|X_t\| \cdot \|X_t^{-1}\| = \tfrac{1}{2} \sqrt{12 + t^2 + \sqrt{16 + 24 t^2 + t^4}} \le \frac{ \sqrt{13 + \sqrt{41}}}{2} \le 2.203.\]
 Thus when $n=3$, our arguments do not provide a proof of Crouzeix's conjecture, though they do give a simple proof of a similar inequality with a slightly worse constant. Meanwhile, in the $n=4,5$ cases, there is no simple formula for $\| X_t \| \cdot \|X_t^{-1}\|.$ However, it can be easily estimated via mathematical software such as Mathematica. These estimations reveal that 
 \begin{itemize}
 \item If $n=4$ and $t \in (0, 0.42)$, $\| X_t \| \cdot \|X_t^{-1}\| \le 2$ so Crouzeix's conjecture holds for $S_{\Theta_t}$. 
 \item  If $n=5$ and  $t \in (0.0001, 0.5)$,  $\| X_t \| \cdot \|X_t^{-1}\| \le 2$ so Crouzeix's conjecture holds for $S_{\Theta_t}$. 
\end{itemize}

These investigations motivate several questions, particularly connected to when pseudohyperbolic disks of certain sizes are contained in numerical ranges of compressions of the shift. Those questions and some accompanying examples appear in Section \ref{sec:examples}.

\section{Level Sets of Finite Blaschke Products} \label{sec:LS}

In this section, we catalog several elegant facts about the structure of level sets of finite Blaschke products. For $r \in (0,1)$ and a finite Blaschke product $B$, define its open level set
\begin{equation} \label{eqn:omegar} \Omega^B_r := \{ z \in \mathbb{C}: |B(z)| <r\} \end{equation} 
and observe that its boundary $\partial \Omega^B_r$ is exactly 
\begin{equation} \label{eqn:Sr} S^B_r := \{ z \in \mathbb{C}: |B(z)| = r\}.\end{equation}
The following lemma characterizes the number of components of $\Omega^B_r$.  The proof follows the same line of argument as the proof of Proposition 2.1 in \cite{Eben2011}.

\begin{lemma} \label{lem:ll} Let $B$ be a finite Blaschke product with $\deg B =m$, let $r \in (0,1)$, and assume $\Omega^B_r$ contains $k$ critical points of $B$. Then $\Omega^B_r$ has $m-k$ components. 
\end{lemma}

\begin{proof} Let $\zeta_1, \dots, \zeta_{m-1}$ denote the critical points of $B$ in $\mathbb{D}$. First, assume that $r \ne |B(\zeta_i)|$ for any $i$. 
If we set $F(x,y) = |B(x+iy)|^2-r^2$, then our assumptions about $r$ imply $\frac{\partial F}{\partial x}$, $\frac{\partial F}{\partial y}$ cannot
simultaneously vanish at any points of the boundary $\partial \Omega^B_r = S^B_r$ and an application of the implicit function theorem says $S^B_r$ is a union of smooth, simple closed curves.

Assume $\Omega^B_r$ contains $k$ critical points of $B$ and has $J$ components denoted by $\Omega_r^1, \dots, \Omega_r^J$. We will show that $J =m-k$. For each $j$, let $k_j$ denote the number of critical points of $B$ in $\Omega_r^j$ and $d_j$ denote the number of zeros of $B$ in $\Omega_r^j$, both counted according to multiplicity. Then
\[ \sum_{j=1}^J k_j = k  \ \ \text{ and }  \ \    \sum_{j=1}^J d_j = m.\]
Define $f_j = B|_{\Omega^j_r}$. Then $f_j:  \Omega^j_r \rightarrow D_r(0)$ is a $d_j$-to-$1$ proper analytic map. To see this, fix $w_0\in D_r(0)$ and define $g(z) = w_0$ for all $z$ and $h = f_j$. As the boundary of $\Omega_r^j$ is a smooth, simple closed curve and 
\[ |g(z)| < |h(z)| =r  \ \text{ on } \partial \Omega^j_r,\]
Rouche's theorem implies that $h-g$ or equivalently $f_j -w_0$ has exactly $d_j$ zeros in $\Omega^j_r$. 

Furthermore, as both $\Omega^j_r$ and $D_r(0)$ are simply connected, the Riemann-Hurwitz theorem in this setting (see for example \cite{N93}) implies that 
\[ -1 = -d_j + k_j. \]
Adding this equation over the $J$ components of $\Omega^B_r$ gives $-J = -m +k$, which is equivalent to the desired formula: $J = m-k$.

Now if $r = |B(\zeta_i)|$ for some $i$, we can still establish the conclusion of the lemma. Specifically, assume that $\Omega^B_r$ contains $k$ critical points of $B$. Then, there is some $\epsilon_0>0$ such that for all $0< \epsilon <\epsilon_0$, $\Omega^B_{r-\epsilon}$ contains exactly the same critical points of $B$ as $\Omega^B_r$. By the previous case, each $\Omega^B_{r-\epsilon}$ 
has $J:= m-k$ components for all $0< \epsilon <\epsilon_0$. Let $(\epsilon_n)$ be a decreasing sequence with each $0< \epsilon_n < \epsilon_0$ and $(\epsilon_n) \rightarrow 0$. For each $n$, number the components $\Omega^1_{r-\epsilon_n}, \dots, \Omega^J_{r-\epsilon_n}$ so that if $n_1<n_2$, then each $\Omega^j_{r-\epsilon_{n_1}} \subseteq \Omega^j_{r-\epsilon_{n_2}}.$ Define
\[U_r^{j} := \bigcup_{n=1}^{\infty} \Omega_{r-\epsilon_n}^j \text{ for each } j=1, \dots, J.\] 
Then by their nested property, one can show that each $U_r^j$ is open and connected, while if $j\ne j'$, then $U_r^{j} \cap \overline{U_r^{j'}}=\emptyset.$ Since 
\[ \Omega^B_r = \bigcup_{j=1}^J U_r^{j},\]
these $U_r^j$ must be exactly the components of $\Omega^B_r$. Thus, $\Omega^B_r$ has exactly $m-k$ components. 
\end{proof}

The following simple observation will be used without further comment in later proofs:

\begin{remark} \label{rem:zeros} Let $r \in(0,1)$ and $B$ be a finite Blaschke product.  Then each component of the level set $\Omega^B_r$ must contain at least one zero of $B$. Indeed, since $|B| \equiv r$ on the boundary $\partial \Omega^B_r$, this conclusion follows immediately from the minimum modulus principle.\end{remark}

For the next level set results, we require the following theorem. This generalizes a well known result of Horwitz and Rubel in \cite{HorwitzRubel86} and may be of independent interest.

\begin{theorem}\label{HRrevised} Let $A$ and $B$ be two monic Blaschke products of degree $n$. Suppose that there are $n$ points $\lambda_1, \ldots, \lambda_n \in \mathbb{D}$ such that $A(\lambda_j) = B(\lambda_j)$ for $j = 1, \ldots, n$, counting multiplicities. Then $A = B$. \end{theorem}

For clarity, recall that two holomorphic functions $f$ and $g$ agree at $\lambda$ with multiplicity $k$ if $f-g$ has a zero of multiplicity $k$ at $\lambda$. Theorem \ref{HRrevised} differs from the classical Horwitz-Rubel theorem in that their result requires $\lambda_1, \dots, \lambda_n$ to be distinct, while this result allows the $\lambda_j$ to be repeated. Still, the proof is similar to the proof of the original theorem and so, we postpone it to Section \ref{sec:RH}.

Returning to level sets, in \cite{Stephenson, Stephenson2} Stephenson and Sundberg studied $r$-level curves of analytic functions  $f$, i.e. curves for which the modulus of $f$ is a constant $r$. For example, they characterized when two inner functions share an $r$ level curve:

\begin{theorem}[Stephenson and Sundberg, \cite{Stephenson2}] Let $f_1$ and $f_2$ be inner functions and suppose that they have an $r$-level curve in common for some $r$ with $0<r<1$. Then there exists $\lambda \in \mathbb{T}$ such that $f_1 = \lambda f_2$. \end{theorem}

Here we use Theorem \ref{HRrevised} to prove a related result for finite Blasche products. While we consider a more restricted class of functions, we only require that their level sets satisfy a containment relationship. 

\begin{theorem} \label{thm:LSB} Let $B$ and $C$ be finite Blaschke products with $\deg B = \deg C$. If there is some $r \in (0,1)$ with $\Omega_r^B \subseteq \Omega_r^C$, then there exists $\lambda \in \mathbb{T}$ such that  $B = \lambda C$.
\end{theorem}

\begin{proof} Without loss of generality, we can assume $B$ and $C$ are monic. Let $\deg B =n = \deg C$ and write $B = \tfrac{q_b}{p_b}$ and $C = \tfrac{q_c}{p_c}$ for polynomials $q_b, p_b, q_c, p_c$ with $\deg q_b = \deg q_c = n$.  By way of contradiction, assume $C \not \equiv B$. Set $f = C-B = \tfrac{r}{p}$ for $p = p_b p_c$ and for each $m \in \mathbb{N}$, define $f_m$ and $r_m$ so
\[ f_m = C - (1 + \tfrac{1}{m})B = \frac{ q_ c p_b - (1+\tfrac{1}{m}) q_b p_c}{p_bp_c} := \frac{r_m}{p}.\]
Then both $f_m \rightarrow f$,  $r_m \rightarrow r$ uniformly on $\overline{\mathbb{D}}$. By the Cauchy integral formula, the derivatives $f_m^{(k)} \rightarrow f^{(k)}$,  $r_m^{(k)} \rightarrow r^{(k)}$ converge uniformly on $\overline{\mathbb{D}}$ as well for each $k \in \mathbb{N}.$ 

Let $\Omega_r^1,\dots, \Omega_r^J$ denote the components of $\Omega_r^B$ and for each $j$, let $n_{B,j}$ denote the number of zeros of $B$ in $\Omega_r^j$. Then $\sum_j n_{B,j} =n$. For each $z \in \partial \Omega_r^j$, 
\[ | f_m(z) +(1 + \tfrac{1}{m})B(z)| = |C(z)| \le r <r (1 + \tfrac{1}{m})  = |(1+ \tfrac{1}{m}) B(z)|.\]
For each $j$ and $m$, shrink $\Omega_r^j$ slightly to obtain a compact $K^j_m \subseteq \Omega_r^j$ such that 
\[  | f_m(z) +(1 + \tfrac{1}{m})B(z)|  < |(1+ \tfrac{1}{m}) B(z)|\]
for each $z \in \partial K^j_m$ and $K^j_m$ contains $n_{B,j}$ zeros of $B$. By a standard version of Rouch\'e's theorem (for example, Theorem $11$ in \cite{MortiniRupp2014}), $B$ and $f_m$ have the same number of zeros in $K_m^j$. Thus, $f_m$ has at least $n_{B,j}$ zeros in each $\Omega_r^j$ and thus, at least $n$ zeros in $\Omega_r^B$. Call these zeros $a_1(m), \dots, a_n(m).$ By passing to a subsequence $(f_{m_\ell})$, we can further assume
\[  a_1(m_\ell) \rightarrow a_1, \dots, a_n(m_\ell) \rightarrow a_n\]
for some $a_1, \dots, a_n \in \overline {\Omega_r^B}$.   

Now we need to show that $\prod_{k=1}^n (z-a_k)$ divides $r$, the numerator of $f$. To that end, observe that for each $\ell$, there is a polynomial $Q_\ell$ with $\deg Q_\ell \le n$ such that the numerator of $f_{m_\ell}$ is given by
 \[ r_{m_\ell}(z)  = Q_\ell(z) \prod_{k=1}^n (z - a_k(m_{\ell})).\]
 Then because each $ a_k(m_{\ell}) \in \Omega_r^B$, for all $z$ with $|z|=2$, we have
 \[ |Q_\ell(z)| \le  \max_{\{z: |z|=2\}} \left( |q_c(z) p_b(z)| + 2|q_b(z) p_c(z) |\right) \prod_{k=1}^n \frac{1}{2 -  |a_k(m_\ell)|} \le M,\]
where $M$ is independent of $\ell$. Thus, $Q_\ell$ and all of its derivatives (via the Cauchy integral formula) are uniformly bounded on $\overline{\mathbb{D}}$ by a constant independent of $\ell$. Now note that for each $k$,
\[ 0 \le  \lim_{\ell \rightarrow \infty} |r_{m_\ell}(a_k)| \le M \lim_{\ell \rightarrow \infty} \prod_{i=1}^n \left| a_k - a_i(m_{\ell}) \right|=0,\]
so $r(a_k)  = 0$. If the list $a_1, \dots, a_n$ contains a repeated zero, say $a$ with multiplicity $s$, then analogous arguments shows that $r(a)=0, r'(a)=0, \dots, r^{(s-1)}(a)=0.$ This implies that $(z-a)^s$ must divide $r$ and putting these together,  $f = \tfrac{Q}{p} \prod_{k=1}^n (z-a_k) $ for some polynomial $Q$ with $\deg Q \le n$. Thus, $f = C -B$ has at least $n$ zeros in $\mathbb{D}$ including multiplicity and hence, Theorem \ref{HRrevised} implies that $B=C$. 
\end{proof}

This theorem implies a similar result if $\deg C \le \deg B$. 

\begin{corollary} \label{cor:LSB} Let $B$ and $C$ be finite Blaschke products with $\deg C \le \deg B$. If there is some $r \in (0,1)$ with $\Omega_r^B \subseteq \Omega_r^C$, then there exists $\lambda \in \mathbb{T}$ such that  $B = \lambda C$.
\end{corollary}

\begin{proof} First we assume that $\deg C < \deg B$ and show that this gives a contradiction. To that end, let $A$ be an arbitrary finite Blaschke product with $\deg A = \deg B -\deg C=:n >0$. Set $\widehat{C} = A C.$ Then if $|C(z)|<r$, it must be the case that $z \in \mathbb{D}$ and so we have 
\[ | \widehat{C}(z) |  =|A(z)|  | C(z) |  < |C(z)| <r.\]
This implies $\Omega_r^C \subseteq \Omega_r^{\widehat{C}}$ and thus by assumption,
\[ \Omega_r^B \subseteq \Omega_r^C \subseteq \Omega_r^{\widehat{C}}.\]
Since $\deg B = \deg \widehat{C}$, Theorem \ref{thm:LSB} implies that there is a constant $\lambda \in \mathbb{T}$ such that $ B = \lambda \widehat{C} = \lambda A C.$ Since $A$ was an arbitrary Blaschke product of degree $n$, this gives a contradiction.  

 Thus, it must be the case that $\deg C = \deg B$. Then the conclusion follows immediately from Theorem \ref{thm:LSB}. \end{proof}

\begin{remark} Theorem \ref{thm:LSB} and Corollary \ref{cor:LSB} are results of the following flavor: if two Blaschke products $B$ and $C$ share a degree inequality  and associated sets share a containment relationship, then the two Blaschke products are equal up to a unimodular constant.  In these cases, the associated sets are $r$-level sets of the Blaschke products.  Gau and Wu (see Lemma \ref{GW} below) showed that a similar result holds if one takes the associated sets to be numerical ranges of compressions of shifts defined using the Blaschke products. These complementary results suggest that level sets of finite Blaschke products and numerical ranges of compressions of shifts may possess some similar structures.
\end{remark}

\begin{remark} Let  $J_{n + 1, a}$ denote the perturbed Jordan block given by \eqref{perturbed} and assume that $ |a| < 1$. These matrices have all eigenvalues in $\mathbb{D}$, they are contractions, and have defect index $1$ (that is, the rank of $I - A^\star A$ is one). Therefore, these represent compressions of the shift operator and one can show that their eigenvalues are the zeros of the function $(z^n-a)/(1 - \overline{a} z^n)$. In \cite{GreenbaumChoi} it is shown that the Crouzeix conjecture holds for such perturbed Jordan blocks, and therefore so does the LSC conjecture. Modifying the proof for $n > 6$ and applying Corollary~\ref{cor:LSB} allows us to view the proof through the lens of level sets.

Recall that $W(S_{z^n}) = W(J_n) = \overline{D(0, \cos(\pi/(n+1))}$ (see \cite{HaagerupdelaHarpe}, for example). Since the Jordan block $J_n$ is a compression of this matrix, $W(J_n) \subseteq W(J_{n+1, a})$. We first show that LSC holds for Blaschke products of the form $\lambda z^n$, where $\lambda \in \mathbb{T}$ and then obtain the desired result from this.

Fix $n \ge 6$ and let $C$ be an arbitrary Blaschke product with $\deg C \le n$.
Then $W(S_{z^n}) = W(J_n) \subseteq W(J_{n+1, a})$. But, $n\ge 6$ implies $\cos(\pi/(n+1))^n > 1/2$.  In particular, $|z^n| > 1/2$ on $\partial W(J_n) \subset W(J_{n+1,a})$, establishing the claim.

Now, if $|C| < 1/2$ on $W(J_{n+1, a})$, then we have $\Omega_{1/2}^{z^n} \subseteq W(S_{z^n}) \subseteq \Omega_{1/2}^{C}$. By Corollary~\ref{cor:LSB}, $C = \lambda z^n$ for some $\lambda \in \mathbb{D}$, and the result holds by the previous paragraph. \end{remark}

\section{LSC Inequality for classes of $B$ and $\Theta$} \label{sec:LSC}

In this section, we use a variety of approaches and techniques to prove that the LSC inequality \eqref{eqn:LSC} holds for several classes of finite Blaschke products $B, \Theta$. One approach that arises frequently is the analysis of related Euclidean and pseudohyperbolic disks. Before proceeding, we establish some notation and a few important formulas.

Let $D_R(c)$ denote a Euclidean disk in $\mathbb{D}$ of radius $R$ and center $c$ and let $D_\rho(z_0, r)$ denote the pseudohyperbolic disk with (pseudohyperbolic) center $z_0$ and pseudohyperbolic radius $r$.
 An important fact is that every pseudohyperbolic disk is a Euclidean disk in $\mathbb{D}$ and every Euclidean disk in $\mathbb{D}$ is a pseudohyperbolic disk. Converting between the two representations is straightforward; first, $D_\rho(z_0, r)$ coincides with the Euclidean disk $D_R(c)$, where 
\begin{equation} \label{eqn:phd1} c = \frac{ (1-r^2)z_0}{1-r^2|z_0|^2} \ \text{ and } R = \frac{r(1-|z_0|^2)}{1-r^2|z_0|^2}.\end{equation}
Meanwhile, if one starts with a Euclidean disk $D_R(c)$, it agrees with the disk $D_\rho(z_0,r)$ where $z_0 \in \mathbb{D}$, $r \in [0,1)$ and the associated centers and radii satisfy the equations 
\[ c = z_0(1-rR) \ \text{ and } R= r(1- |c| |z_0|),\]
see for example page $3$ in \cite{Garnett}.
More specifically, if $c = 0$, then $D_R(c) = D_\rho(0,R)$. If $c \ne 0$, then $D_R(c)$ coincides with $D_\rho(z_0, r)$ where $\arg z_ 0 = \arg c$, $|z_0|$ is the unique solution in $[0,1)$ of 
\begin{equation} \label{eqn:z0}  |z_0| + \tfrac{1}{|z_0|} = \frac{ |c|^2 -R^2 +1}{|c|},\end{equation}
and $r$ is the unique solution in $[0,1)$ of 
\begin{equation} \label{eqn:r}  r + \tfrac{1}{r} = \frac{R^2 -|c|^2 +1}{R}.\end{equation}
These formulas can be found in \cite{MortiniRupp2021}.

\subsection{LSC Inequality via Pseudohyperbolic Disks} We first establish the LSC inequality \eqref{eqn:LSC}  when $W(S_\Theta)$ contains a sufficiently large pseudohyperbolic disk. 
Specifically, Corollary \ref{cor:LSB} leads to the following result:

\begin{theorem}\label{thm:main} Let $B$ be a finite Blaschke product with $\deg B \le m$. Then for each $z_0 \in \mathbb{D}$ and $r \in (0,1)$,
\begin{equation} \label{eqn:Bdisk} \sup \{ |B(z)| : z\in D_\rho(z_0, r^{1/m})\} \ge r.\end{equation}
\end{theorem}

\begin{proof} Let $C$ be a unicritical Blaschke product with $\deg C = m$ and its zero at $z_0$. Then 
\[ \Omega^C_r =\left \{ z\in \mathbb{C}: \left | \frac{z-z_0}{1-\bar{z}_0 z}\right|^m < r \right\} = D_\rho(z_0, r^{1/m}).\]
 By way of contradiction, assume
\[ \sup \{ |B(z)| : z\in D_\rho(z_0, r^{1/m})\} < r.\]
This assumption implies that $\Omega^C_r \subseteq \Omega^B_r$. Then as $\deg B \le \deg C$, Corollary \ref{cor:LSB} implies that there is some $\lambda \in \mathbb{T}$ with $B= \lambda C$. But then
\[ \sup \{ |B(z)| : z\in D_\rho(z_0, r^{1/m})\}=\sup \{ |C(z)| : z\in D_\rho(z_0, r^{1/m})\} = r,\]
a contradiction. Thus, \eqref{eqn:Bdisk} must hold. \end{proof} 

Theorem \ref{thm:main} gives the following corollary related to numerical ranges:

\begin{corollary} \label{cor:main1}  Let $A$ be a square matrix and  $B$ a finite Blaschke product with $\deg B  \le m$. If there is a pseudohyperbolic disk $D_\rho(z_0, r^{1/m}) \subseteq W(A)$, then
\[  \sup \{|B(z)|: z \in W(A) \cap \text{Domain}(B) \} \ge r.\] 
\end{corollary}

By restricting to compressions of the shift, this also gives the LSC inequality \eqref{eqn:LSC} for $\Theta$ whose associated numerical ranges contain a large enough pseudohyperbolic disk.

\begin{corollary}  \label{cor:main}  Let $B, \Theta$ be finite Blaschke products with $\deg B < \deg \Theta:=n$.  If  there is a pseudohyperbolic disk $D_\rho(z_0, (\tfrac{1}{2})^{1/(n-1)}) \subseteq W(S_\Theta)$, then 
\[  \max \{|B(z)|: z \in W(S_\Theta)  \} \ge \tfrac{1}{2}.\] 
\end{corollary} 

Later, we use Corollary \ref{cor:main} to study unicritical $\Theta$ in Section \ref{sec:uniTheta} and will provide some associated examples in Section \ref{sec:examples}.

\subsection{LSC Inequality via Fuss's Formula} In this section, we use Fuss' formula for circles circumscribed by quadrilaterals that are inscribed in a circle to prove \eqref{eqn:LSC} when $\deg B =2$ and $\deg \Theta \ge 6$. 

Our proof will also use the following result of Gau and Wu:

\begin{lemma}[Lemma 4.2, \cite{GauWu}\label{GW}] 
Let $S_{\Theta_1}$ and $S_{\Theta_2}$ denote two compressions of the shift with $\Theta_1$ and $\Theta_2$ Blaschke products with $\deg \Theta_1 \le \deg \Theta_2$. Then $S_{\Theta_1}$ is unitarily equivalent to $S_{\Theta_2}$ if and only if $W(S_{\Theta_1})$ contains $W(S_{\Theta_2})$. \end{lemma}

\c
For the purposes of this paper, a Poncelet $(n+1)$-ellipse in $\mathbb{D}$ is an ellipse that is inscribed in a convex $(n+1)$-gon that is itself inscribed in the unit circle. Every Poncelet $(n + 1)$-ellipse in $\mathbb{T}$ is the boundary of the numerical range of a compression of the shift operator $S_\Theta$, where $\Theta$ is a Blaschke product of degree $n$. (This is stated in general in \cite{GauWu}, \cite[p. 219]{GauWuCond} and the proof for $n = 2$ this can be found in the discussion of Conjecture 5.1 of the same paper. For $n = 3$, this appears in \cite[Corollary 3.8]{GorkinWagner}.) We use this in Lemma \ref{lem:fuss} below and the remark that follows it.

\begin{lemma} \label{lem:fuss} Fix $a \in \mathbb{D}$. Then there exists a Blaschke product $\Psi$ with $\deg \Psi = 3$ such that $W(S_\Psi)$ equals the  closure of the  pseudohyperholic disk $D_\rho(a,r)$ for some $r \ge \tfrac{1}{\sqrt{2}}$. \end{lemma}

\begin{proof} Fix $c \in \mathbb{D}$. By Fuss' formula (see \cite{Fuss} or \cite[Corollary 2]{Fujimura} for a modern reference) there is exactly one disk centered at $c$ whose boundary is inscribed in a quadrilateral that is circumscribed by $\mathbb{T}$ (and thus, is a so-called Poncelet-$4$ circle) and this disk, which we denote $D_R(c)$, has radius
\[ R = \frac{1-|c|^2}{ \sqrt{ 2(1 + |c|^2)}}.\] 
By substituting directly into the formulas from \eqref{eqn:phd1}, one can check that this Euclidean disk agrees with the pseudohyperbolic disk $D_\rho(\tilde{a}, r)$ with pseudohyperbolic center and radius given by
\[ \tilde{a} = \frac{2c}{1+|c|^2} \ \text{ and } r = \frac{\sqrt{1 +|c|^2}}{\sqrt{2}}.\]
Now consider the $a$ in the statement of the lemma and choose $c \in \mathbb{D}$ so that $\text{Arg}(c) = \text{Arg}(a)$ and $|a| = \frac{2 |c|}{1 +|c|^2}$. Then, the above arguments imply that $D_R(c) = D_\rho(a, r)$, where $r \ge 1/\sqrt{2}.$  By the remarks preceding this lemma and our assumption that the Euclidean disk is bounded by a Poncelet $4$-circle, there is a finite Blaschke product $\Psi$ with $\deg \Psi = 3$ and $W(S_\Psi)$ equal to the closure of $D_\rho(a, r)$.
\end{proof}

\begin{theorem} \label{thm:26} Let $B, \Theta$ be finite Blaschke products with $\deg B = 2$ and $\deg \Theta \ge 6$. Then 
\begin{equation} \label{eqn:cc26}  \max_{z \in W(S_\Theta)} |B(z)| \ge \tfrac{1}{2}. \end{equation}
\end{theorem}

\begin{proof} Let $a_1, a_2$ denote the zeros of $B$. It is easy to see that $\Omega^B_{1/2}$ satisfies the containment property
\[ \Omega^B_{1/2}  \subseteq D_\rho(a_1, \tfrac{1}{\sqrt{2}}) \cup D_\rho(a_2, \tfrac{1}{\sqrt{2}}).\]
By Lemma \ref{lem:fuss}, there exist finite Blaschke products $\Psi_1, \Psi_2$ with $\deg \Psi_j =3$ and $r_j \ge \tfrac{1}{\sqrt{2}}$ for $j=1,2$ such that
\[ W(S_{\Psi_j} ) = D_\rho(a_j , r_j) \supseteq D_\rho(a_j , \tfrac{1}{\sqrt{2}}).\]
 Set $\widetilde{\Theta} = \Psi_1 \Psi_2.$ Then 
\[  \Omega^B_{1/2} \subseteq W(S_{\Psi_1} ) \cup W(S_{\Psi_2} ) \subseteq W(S_{\widetilde{\Theta} } ).\]  
 There are two cases. If $\Theta = \lambda \widetilde{\Theta}$ for some $\lambda \in \mathbb{T}$, then since $W(S_\Theta)$ is closed, we actually know
 $\overline{\Omega}^B_{1/2}$ is contained in $W(S_\Theta)$, and so \eqref{eqn:cc26} holds. If $\Theta \ne \lambda \widetilde{\Theta}$ for any $\lambda \in \mathbb{T}$,
 then since $S_{\Theta}$ is not unitarily equivalent to $S_ {\widetilde{\Theta}}$ and deg $\Theta \ge$ deg $\widetilde{\Theta}$, Lemma~\ref{GW} implies that there is some $z_0 \in W(S_\Theta) \setminus  W(S_{\widetilde{\Theta} } ).$
By the given set containments, this  implies $|B(z_0)| \ge \tfrac{1}{2}$, so again \eqref{eqn:cc26} holds. 
\end{proof}

\begin{remark} This result can be improved if $|a|$ is close enough to $1$: A Poncelet $3$-circle is the boundary of $W(S_\varphi)$ for some Blaschke product $\varphi$ of degree-$2$. By the Chapple-Euler formula (see \cite{DGM} or \cite[p. 197]{DGSV}) this circle has equation \[|z - c| = (1-|c|^2)/2.\] Using \eqref{eqn:phd1}, we find that this Euclidean circle has pseudohyperbolic radius \[r = \frac{5 - |c|^2 - \sqrt{9 - 10 |c|^2 + |c|^4}}{4}.\] Solving for $|c|$ and then for $|a|$ shows that if \[|a| \ge \frac{-5 + 6 \sqrt{2} - \sqrt{17 - 12 \sqrt{2}}}{4 \sqrt{5 - 3 \sqrt{2}}},\] then $r \ge 1/\sqrt{2}$. So for such $a$, we can take $\Psi$ to be of degree $2$. In particular, if both $a_1$ and $a_2$ have modulus close enough to $1$, we may assume that $\deg \Theta \ge 4$ in Theorem~\ref{thm:26}. Furthermore, formulas for Poncelet $n$-circles exist, but they are not easy to work with. (See \cite[p. 197]{DGSV}.)
\end{remark}

\subsection{LSC Inequality via Zero Set Conditions}

The results in this section should be compared to those in \cite[Corollary 2.3]{BGGRSW}. We first establish the following:

\begin{theorem}\label{thm:hoffman} Let $B, \Theta$ be finite Blaschke products such that $B$ satisfies $B(0) = 0$ and $|B'(0)| \ge  \tfrac{2\sqrt{2}}{3} \approx 0.94$  and $\Theta$ satisfies $\deg \Theta \ge 9$ and  $0 \in W(S_\Theta)$. Then 
\[\max_{z \in W(S_\Theta)} |B(z)| \ge \tfrac{1}{2}.\]
\end{theorem}

\begin{proof} We modify an argument of K. Hoffman (\cite[Lemmas 4.1 and 4.2]{Hoffman}) to obtain this result. 

\bigskip
Suppose that $B(0) = 0$ and let $h(z) := B(z)/z,$ so $h(0) = B^\prime(0)$. Applying the Schwarz-Pick lemma to $h$ gives
\begin{equation}
\label{eqn:Schwarz}
\rho(h(z), h(0)) \le \rho(z, 0) = |z|.
\end{equation}
Therefore (see \cite[p. 4]{Garnett})
\[|h(z)| \ge \rho(h(z), 0) \ge \frac{\rho(h(0), 0)-\rho(h(z), h(0))}{1- \rho(h(z), h(0)) \rho(h(0), 0)}  = \frac{|h(0)|-\rho(h(z), h(0))}{1-|h(0)| \rho(h(z), h(0))}.\]

Now if $a\in (-1,1)$, the function $(a-x)/(1-a x)$ is a decreasing function of $x$, so equation~\eqref{eqn:Schwarz} implies that

\[|h(z)| \ge \rho(h(z), 0) \ge  \frac{|h(0)|-\rho(h(z), h(0))}{1-|h(0)| \rho(h(z), h(0))} \ge \frac{|h(0)| - |z|}{1 - |h(0)| |z|}.\] 
Set $\delta = |B'(0)|$. Since $h(z) = B(z)/z$ we have
\[|B(z)| \ge \frac{|B^\prime(0)| - |z|}{1 - |B^\prime(0)| |z|} |z| =  \frac{\delta - |z|}{1-\delta |z|}{|z|}.\]

We are interested in when this is greater than or equal to $1/2$.  We note that  $\frac{\delta - x}{1-\delta x}{x}$ has a maximum when $|z| =x= \frac{1 - \sqrt{1-\delta^2}}{\delta}$ and the value is
 \[\frac{\delta - \frac{1 - \sqrt{1-\delta^2}}{\delta}}{1 - \delta \frac{1 - \sqrt{1-\delta^2}}{\delta}} \frac{1 - \sqrt{1-\delta^2}}{\delta}  = \left(\frac{1-\sqrt{1-\delta^2}}{\delta}\right)^2.\] A computation shows that the maximum is greater than or equal to $1/2$ when $\delta \ge 2\sqrt{2}/3$.  So $|B(z)| \ge 1/2$ when $|z| = x = \frac{1 - \sqrt{1-\delta^2}}{\delta}$ and $\delta \ge  2\sqrt{2}/3$.  Using the fact that the numerical range of a $9 \times 9$ Jordan block is the numerical range of $S_{z^9}$ and $W(S_{z^9})$ is the closed disk $\overline{D_{\cos(\pi/10)}(0)}$,  we apply Lemma~\ref{GW} to conclude that  $W(S_\Theta)$ cannot be contained in this circle of radius $\cos(\pi/10) >  2\sqrt{2}/3$. Since $0 \in W(S_\Theta)$,  there is a point in $W(S_\Theta)$ with modulus greater than $ 2\sqrt{2}/3$, and since $W(S_\Theta)$ is convex, there must be a point in the numerical range (on the circle $|z| =  \frac{1 - \sqrt{1-\delta^2}}{\delta}$) where $|B(z)| \ge 1/2$.
\end{proof}

We now show that we can drop the condition that $B(0)=0$ to conclude that the LSC inequality \eqref{eqn:LSC} holds for $B$ with sufficiently separated zeros and $\Theta$ with a sufficiently large associated numerical range:

\begin{corollary} \label{cor:zerosep} Let $B$ be a Blaschke product with zeros $a_1, \ldots, a_n$ satisfying $\prod_j |a_j| \ge \tfrac{2\sqrt{2}}{3} $ and let $\Theta$ be a Blaschke product with $\deg \Theta\ge 9$  and $0 \in W(S_\Theta)$. Then 
\[\max_{z \in W(S_\Theta)} |B(z)| \ge \tfrac{1}{2}.\]\end{corollary}

\begin{proof} In what follows, we let $\hat{B}(z) = z B(z)$. Then $\hat{B}(0) = 0$ and $|\hat{B}^\prime(0)| = \prod_j |a_j| \ge \tfrac{2\sqrt{2}}{3} $. As $|B(z)| > |\hat{B}(z)|$ on $\mathbb{D}$, the result follows from Theorem \ref{thm:hoffman}.
 \end{proof}

\begin{remark} It is worth noting that, while the assumption $0 \in W(S_\Theta)$ is often a natural one in the study of numerical ranges, it is somewhat arbitrary here. Indeed we just need to assume that $ W(S_\Theta)$ contains some point $z_0$ with $|z_0| <\tfrac{2\sqrt{2}}{3} $.
\end{remark}

\subsection{LSC Inequality via Level Set Components} 
In this section, we establish the LSC inequality \eqref{eqn:LSC} for every degree-$2$ Blaschke product $B$ such that $\Omega_{1/2}^B$ has two components and any $\Theta$ with $\deg \Theta >2.$ We first require some preliminary information about the structure of such two-component level sets. 

\begin{lemma} \label{lem:inequality} Let $B$ be a degree-$2$ Blaschke product with distinct zeros $a_1, a_2$.  Let $\zeta$ be the
critical point of $B$ in $\mathbb{D}$ and choose $r$ with $0 <r< |B(\zeta)|$. Then $\Omega^B_r$ has two components $\Omega_r^1, \Omega_r^2$ with $a_j \in \Omega_r^j$ and for all $z\in \Omega_r^1$, $\rho(z,a_1) < \rho(z,a_2).$
\end{lemma}

\begin{proof} By Lemma \ref{lem:ll}, $\Omega^B_r$ has two components $\Omega_r^1, \Omega_r^2$.  By Remark \ref{rem:zeros}, we can assume $a_j \in \Omega_r^j$ for each $j$. 

 We first consider the special case $a_1 = 0$ and $a_2=t \in (0,1).$  We will show: for all $z\in \Omega_r^B$, we have $\rho(z, 0) \ne \rho(z, t).$ To establish that, it suffices to show that for all $z\in \Omega_r^1$, we have $\rho(z, 0) < \rho(z, t).$ 
 
Note that $\Omega^B_r \subseteq D_\rho(0,r^{1/2}) \cup D_\rho(t,r^{1/2})$. We will show that these two pseudohyperbolic disks are disjoint. To that end, consider the family of pseudohyperbolic disks,  $D_\rho(0,s), D_\rho(t,s)$ for $s \in(0,1)$. As these are also Euclidean disks, we can 
let $c_1, c_2$ and $R_1, R_2$ denote the Euclidean centers and radii of $D_\rho(0,s), D_\rho(t,s)$ respectively. Their values are given by
 \[ c_1 =0, c_2 = \frac{(1-s^2)t}{1-s^2t^2} \text{ and } R_1 = s, R_2 =\frac{s(1-t^2)}{1-s^2t^2}.\] 
 As $s$ increases from $0$ to $1$, the disks $D_\rho(0,s), D_\rho(t,s)$ are initially disjoint, then tangent, and then intersect. By standard properties of circles, they are tangent exactly when
 \begin{equation}\label{eqn:tangent} |c_1 - c_2|^2  = (R_1 +R_2)^2.\end{equation}

 Then solving \eqref{eqn:tangent} for $s$ shows that those circles are tangent exactly when $s = \tilde{s}:= \tfrac{1-\sqrt{1-t^2}}{t}$ and do not intersect for smaller $s$. It is easy to check that $\tilde{s}^2 = |B(\zeta)|$. Then by assumption, $r^{1/2} <\tilde{s}$ and so, $ D_{\rho}(0,r^{1/2}) \cap D_\rho(t,r^{1/2})=\emptyset.$ Since $0 \in \Omega_r^1$, if $z \in \Omega_r^1$, then $z\in D_\rho(0, r^{1/2})$ and we have
 \[ \rho(z,0) < r^{1/2} \le \rho(z,t).\]
 More generally, this argument shows that $\rho(z,0) \ne \rho(z,t)$ for $z\in \Omega^B_r$. 
 
Now we proceed to the general case. Note that because 
\[ 0 = \rho(a_1,a_1) < \rho(a_1,a_2),\] 
by continuity, we just need to show that $\rho(z,a_1) \ne \rho(z,a_2)$ for $z \in \Omega^B_r$. Let $\phi$ be an automorphism of $\mathbb{D}$ with $\phi(a_1) = 0$ and $t:=\phi(a_2) \in(0,1).$ Define $\widehat{B} = B \circ \phi^{-1}$, so that $B = \widehat{B} \circ \phi$. Then $\phi(\zeta)$ is the critical point of $\widehat{B}$ in $\mathbb{D}$. By way of contradiction, assume $\rho(z,a_1) = \rho(z,a_2)$ for some $z\in \Omega^B_r$. Then
\[ \rho( \phi(z),0) = \rho(\phi(z), \phi(a_1)) = \rho(z,a_1) = \rho(z,a_2) =  \rho(\phi(z), \phi(a_2))  =\rho(\phi(z), t).\]
Because $r < |B(\zeta)| = |\widehat{B}(\phi(\zeta))|$, this contradicts the special case we already established and proves the claim.
\end{proof}

The next section considers the special case in which the Blaschke product $B$ is unicritical. In the following theorem, we present an application of Lemma~\ref{lem:inequality} that uses one of these results. 

\begin{theorem}\label{thm:degree2level} Let $B$ be a degree-$2$ Blaschke product with distinct zeros $a_1, a_2$ and assume $\Omega^B_{1/2}$ has two components. If $\deg \Theta \ge 3$, then 
\[ \max\{ |B(z)|: z \in W(S_\Theta)\} \ge \tfrac{1}{2}.\]
\end{theorem}

\begin{proof} Assume the conclusion does not hold. Then $W(S_\Theta) \subseteq \Omega^B_{1/2}$. Because $W(S_\Theta)$ is connected, without loss of generality, we can assume that $W(S_\Theta) \subseteq \Omega^1_{1/2}$. Then by Lemma \ref{lem:inequality}, 
\[ |C(z)|:= \left | \frac{z-a_1}{1-\bar{a}_1 z}\right|^2 \le |B(z)| \text{ on } \Omega_{1/2}^1\]
 and hence on $W(S_\Theta)$. By Corollary \ref{cor:uniB} below, there is some $z_0 \in W(S_\Theta)$ with $|C(z_0)| \ge \tfrac{1}{2}$. Thus  $|B(z_0)| \ge \tfrac{1}{2}$ as well, which establishes the theorem. 
\end{proof}

\section{The Case of a Unicritical $B$} \label{sec:uniB}

In this section, we consider the LSC inequality \eqref{eqn:LSC}, and more general estimates, in the setting where $B$ is unicritical, i.e.
 \[B(z) = \lambda \left(\frac{z-z_0}{1-\overline{z_0}z}\right)^{m},\] 
 for some $m\ge 1$, $z_0 \in \mathbb{D}$, and $\lambda \in \mathbb{T}$ . 


\begin{theorem}\label{thm:basic} Let $A$ be a square matrix and $B$ a degree-$m$ unicritical Blaschke product with zero $z_0$. Assume $\tfrac{1}{\overline{z_0}} \not \in \sigma(A)$ and  $\| B(A) \| =k < 2$.  Then 
\[  \sup \{|B(z)|: z \in W(A) \cap \text{Domain}(B) \}  \ge \tfrac{k}{2}.\] \end{theorem}
%
%

\begin{proof} As $(\tfrac{k}{2})^{1/m} <1$, by \eqref{eqn:phd1},  there is a Euclidean center $c \in \mathbb{D}$ and radius $R <1$ such that
\[D_\rho(z_0,(\tfrac{k}{2})^{1/m}) = D_{R}(c).\] 
Note that $|B(z) | =\tfrac{k}{2}$ on the boundary $\partial D_{R}(c)$, is strictly less than $\tfrac{k}{2}$ in $D_{R}(c)$, and is strictly greater than $\tfrac{k}{2}$ on $\mathbb{C} \setminus \overline{D_{R}(c)}$ (except at $\tfrac{1}{\overline{z_0}}$ where it is undefined).

By way of contradiction, assume  
\[ \sup \{|B(z)|: z \in W(A) \cap \text{Domain}(B)  \} < \tfrac{k}{2}.\] 
This implies $\tfrac{1}{\overline{z_0}} \not \in W(A)$ and $W(A) \subseteq D_R(c)$, and as $W(A)$ is compact, 
 there must be an $\varepsilon >0$ such that, letting $R_\varepsilon:=(1-\varepsilon)R,$ we have
 \[ W(A) \subseteq D_{R_\varepsilon}(c).\]
By well-known results (see for example the arguments in Proposition $3.4$ in \cite{Crouzeix1} or Section $6$ in \cite{CGL18}), this implies that $ D_{R_\varepsilon}(c)$ is a 
two-spectral set for $A$; that is, for all polynomials $p$,
\[\|p(A)\| \le 2 \sup \{|p(z)|:z\in D_{R_\varepsilon}(c)\}.\]
Since this holds for all polynomials, it immediately extends to all functions in the disk algebra $\mathcal{A}(\mathbb{D})$ and in particular, it holds for $B$.  
Since  $\overline{D_{R_\varepsilon}(c)}$ is strictly contained in $ D_{R}(c)$, it also follows that 
 \[  \sup \{|B(z)|: z \in D_{R_\varepsilon}(c)\} < \tfrac{k}{2}\]
 which, by assumption, gives 
\[k = \|B(A)\| \le 2 \sup \{|B(z)|: z \in D_{R_\varepsilon}(c)\} < 2 \cdot \tfrac{k}{2} = k.\]
This yields the contradiction and establishes the result. \end{proof}

As corollaries, we immediately get the following results for unicritical Blaschke products and  automorphisms  applied to compressions of shifts:

\begin{corollary} \label{cor:uniB}  Let $\Theta$ and $B$ be finite Blaschke products with $\deg B <\deg \Theta$ and  $B$ unicritical with zero $z_0$. Then \[ \max \{|B(z)|: z \in W(S_\Theta) \} \ge \tfrac{1}{2}.\] \end{corollary}

\begin{proof}  As $W(S_\Theta) \subseteq \mathbb{D}$, we know $\tfrac{1}{\overline{z_0}} \not \in W(S_\Theta).$ As discussed earlier, Corollary $4$ in \cite[p. 512]{GarciaRoss} implies $\| B(S_\Theta)\| =1$.  By Theorem~\ref{thm:basic} with $k = 1$, we find that \[  \max \{|B(z)|: z \in W(S_\Theta) \}  \ge \tfrac{1}{2},\]
which is what we needed to show.
 \end{proof}

 \begin{corollary} Let $\Theta$ be a finite Blaschke product with $\deg \Theta=n >1$ and let 
 $\varphi$ be an automorphism of the unit disk. Then 
\[\max\{\left|\varphi(z)\right|: z \in W(S_\Theta)\} \ge \left(\tfrac{1}{2}\right)^{1/(n-1)}.\] \end{corollary}

\begin{proof} By definition, we can write $\varphi(z) :=\lambda \frac{z-z_0}{1-\overline{z}_0 z}$ for some $z_0 \in \mathbb{D}$ and $\lambda \in \mathbb{T}$. Set $B(z) =  (\frac{z-z_0}{1-\overline{z}_0z})^{n-1}$. Then $B$ is unicritical with $\deg B < \deg \Theta$, so by Corollary \ref{cor:uniB}, 
\[  \max \{|B(z)|: z \in W(S_\Theta) \}  \ge \tfrac{1}{2}.\] Therefore, there exists $a \in W(S_\Theta)$ such that $\left|\varphi(a)\right| \ge \left(\tfrac{1}{2}\right)^{1/(n-1)}.$
\end{proof} 

\section{The Case of a Unicritical $\Theta$} \label{sec:uniTheta}
In this section, we consider the LSC inequality \eqref{eqn:LSC} when $\Theta$ is unicritical, i.e.
 \[\Theta(z) =  \lambda \left(\frac{z-z_0}{1-\overline{z_0}z}\right)^{n},\] 
 for some $n\ge 1$, $z_0 \in \mathbb{D}$, and $\lambda \in \mathbb{T}.$ As $\lambda$ does not affect the operator $S_\Theta$, we will typically assume $\lambda=1$.  We will often use the notation $\Theta_{z_0}$ or $\Theta_{z_0}^n$ when we need to keep track of the zero $z_0$ or power $n$. For $\Theta$, establishing \eqref{eqn:LSC} is really a question about the numerical range $W(S_\Theta)$ and thus our initial discussion here focusses on its structure.

A lot is known about the numerical ranges $W(S_{\Theta})$ associated to unicritical $\Theta$. For example in \cite{Gaaya12}, Gaaya characterized their numerical radii and established a number of intermediate results, including the following useful equality in his Proposition 2.6:
\[ W(S_{\Theta_{z_0}})= e^{i \text{arg }( z_0)} W(S_{\Theta_{|z_0|}}).\]
Thus to study $W(S_{\Theta_{z_0}})$, we can generally assume that $z_0 = t\in [0,1)$. 
For $t\in(-1,1)$, let $M_t$ be the matrix representation of $S_{\Theta_t}$ with respect to the 
Takenaka-Malmquist-Walsh basis of $K_{\Theta_t}$ (see pages $114-117$ for a discussion of both this basis and $M_t$ \cite{DGSV}). Then $M_t$
is an upper triangular matrix given by $M_t = t I + (1-t^2) A_t$, where  $I$ is the $n\times n$ identity matrix and $A_t$ is the upper triangular nilpotent matrix
\begin{equation} \label{eqn:At} A_t = \begin{bmatrix} 0 &1 & -t & \dots & (-t)^{n-2} \\
  &  0 & 1 & \ddots & \vdots \\
 & &   & \ddots& -t \\
 &  & &   & 1\\ 
0 & &  &   & 0
\end{bmatrix}. \end{equation}
The matrix $A_t$ is sometimes called a KMS matrix and the numerical ranges of these matrices have been studied by Gau and Wu in \cite{GauWu13, GauWu14}.
In the $3 \times 3$ case, Crouzeix's results from \cite{C13} can be applied to obtain the boundary of $W(A_t)$ and hence, of $W(M_t)$. In particular, following  \cite[p.39]{C13} set
\[ m_t(s) = -\tfrac{2}{\sqrt{3}} \sin \left(\frac{\pi + \text{arcsin}\left(3\sqrt{3} \tfrac{-2t}{2(2+t^2)^{3/2}} \cos(s)\right)}{3}\right),\]
for $t \in (-1,0]$ and $s \in [0, 2\pi]$. Then the formula for the boundary representation from \cite[p.31]{C13}  implies that
the boundary of $W(M_t)$ is parameterized by $(\hat{x}_t(s), \hat{y}_t(s))$, where
\[ \begin{aligned}
\hat{x}_t(s) &=t +\tfrac{(1-t^2)}{2} \sqrt{2+t^2} \left( -\cos(s) m_t(s) + \sin(s)m_t'(s) \right) \\
\hat{y}_t(s) &=\tfrac{(1-t^2)}{2} \sqrt{2+t^2} \left( -\sin(s) m_t(s) - \cos(s)m_t'(s) \right),
\end{aligned}
\]
and a similar formula holds if $t \in (0,1)$. This boundary formula illustrates the fact that even though $A_t$ and $M_t$ appear simple, their numerical ranges are quite complicated. Indeed, it is quite difficult to use this boundary formula to deduce quantitative results about $W(S_{\Theta_t})$. Instead, in the following subsection, we find a useful curve $\mathcal{C}_t$ inside $W(A_t)$.

\subsection{A curve in $W(A_t)$} 

The key result in this section is the following:

\begin{proposition} \label{prop:Ct} Fix $t\in(-1,1)$ and let $A_t$ be the $n\times n$ matrix from \eqref{eqn:At}. Then $W(A_t)$ contains the curve $C_t:=C^n_t$ parameterized by 
\begin{equation} \label{eqn:Ct1} \sum_{k=1}^{n-1} a_{n,k} (-t)^{k-1}  e^{isk}, \qquad  s \in [0, 2\pi),\end{equation}
where $a_{n,1} =  \cos\left( \frac{\pi}{n+1} \right)$ and more generally,
\begin{equation} \label{eqn:ank} a_{n,k} = \frac{1}{(n+1)\sin\left(\frac{\pi}{n+1}\right)}\left( (n-k) \cos\left( \tfrac{k \pi}{n+1}\right)\sin\left(\tfrac{\pi}{n+1}\right) +  \sin\left( \tfrac{\pi(n-k)}{n+1}\right)\right).\end{equation}
%
\end{proposition}
\begin{proof} First note that for all $\vec{x} \in \mathbb{C}^n$, a straightforward computation gives
\begin{equation}\label{eqn:At1} \langle A_t \vec{x}, \vec{x} \rangle = \sum_{k=1}^{n-1} (-t)^{k-1}\sum_{\ell=1}^{n-k} \bar{x}_{\ell} x_{\ell +k}.\end{equation}
Fix $s \in [0, 2\pi)$ and define $\vec{x} \in \mathbb{C}^n$ by 
\[ x_\ell = \sqrt{\tfrac{2}{n+1}} \sin\left( \tfrac{ \ell \pi}{n+1} \right) e^{i(\ell-1)s}, \quad \text{ for } 1 \le \ell \le n.\] 
Then as observed in \cite[Proposition 1]{HaagerupdelaHarpe}, $\| \vec{x}\| =1$ and 
\[ \sum_{\ell=1}^{n-1} \bar{x}_{\ell} x_{\ell +1} = \cos\left( \tfrac{\pi}{n+1} \right) e^{is}.\] 
Substituting that into \eqref{eqn:At1} and factoring out a common $\frac{2}{n+1}$ yields the following point corresponding to $s$:
\begin{equation} \label{eqn:Ct} \cos\left( \tfrac{\pi}{n+1} \right) e^{is} +\tfrac{2}{n+1} \sum_{k=2}^{n-1} (-t)^{k-1} \sum_{\ell=1}^{n-k} \sin\left(\tfrac{\ell \pi}{n+1}\right) \sin\left( \tfrac{(k+\ell)\pi}{n+1}\right) e^{i(ks)}.\end{equation}
Fix $k \ge 2$. Then standard trigonometric identities imply that 
\begin{equation} \label{eqn:Ct2} \sum_{\ell=1}^{n-k} \sin\left(\tfrac{\ell \pi}{n+1}\right) \sin\left( \tfrac{(k+\ell)\pi}{n+1}\right)  = \tfrac{1}{2} (n-k) \cos\left( \tfrac{k\pi}{n+1}\right) -  \tfrac{1}{2} \sum_{\ell=1}^{n-k}  \cos\left(\tfrac{(k+2\ell)\pi}{n+1} \right).\end{equation}
Observe that we can write the second term on the right-hand-side of that equation as 
\[  -\tfrac{1}{2} \sum_{\ell=1}^{n-k}  \cos\left(\tfrac{(k+2\ell)\pi}{n+1} \right)  =  -\tfrac{1}{2} \sum_{\ell=0}^{N-1} \cos \left( \alpha + \ell \beta\right),\]
where
\[ \alpha = \tfrac{(k+2) \pi}{n+1}, \ \ \ \beta = \tfrac{2 \pi}{n+1}, \ \ \ \text{ and }  \ \ N = n-k.\]
Then well-known identities for arithmetic progressions of angles in trigonometric functions, see \cite[p. 371]{knapp}, imply
\[  \sum_{\ell=1}^{n-k}  \cos\left(\tfrac{(k+2\ell)\pi}{n+1} \right) =  \frac{ \sin\left (\tfrac{N\beta}{2}\right)}{\sin\left(\tfrac{\beta}{2}\right)} \cos\left( \alpha + \tfrac{(N-1)\beta}{2} \right) =-\frac{ \sin\left( \tfrac{(n-k)\pi}{n+1}\right)}{\sin\left(\tfrac{\pi}{n+1}\right)}.\]
Substituting that into \eqref{eqn:Ct2} and then \eqref{eqn:Ct} yields the point from \eqref{eqn:Ct1}.
Since each such point is in $W(A_t)$, the curve $C_t$ is also in $W(A_t)$. 
\end{proof}

\begin{remark} \label{rem:ct} While the formula for $C_t$ in Proposition \ref{prop:Ct} appears complicated, it simplifies quite dramatically for  small values of $n$. For example, 
\begin{itemize}
\item If $n=3$, $C_t$ is given by $\tfrac{1}{\sqrt{2}} e^{is} - \tfrac{t}{4} e^{i(2s)}.$
\item If $n=4$, $C_t$ is given by $\tfrac{1}{4}(1 + \sqrt5) e^{is} -\tfrac{t}{\sqrt{5}} e^{i(2s)} + \tfrac{t^2}{4}\left(1-\tfrac{1}{\sqrt{5}}\right) e^{i(3s)}$. 
\item If $n=5$, $C_t$ is given by $\tfrac{\sqrt{3}}{2} e^{is} -\tfrac{7t}{12} e^{i(2s)} +\tfrac{t^2\sqrt{3}}{6} e^{i(3s)} -\tfrac{t^3}{12} e^{i(4s)}.$
\end{itemize}
Figure \ref{fig:Ct} illustrates these curves $C_t \subseteq W(A_t)$  and $t +(1-t^2)C_t \subseteq W(M_t)$ for $n=3, \dots, 9$ and $t=0.8.$ These curves do not (in general) appear to be convex, but 
they do appear to grow as $n$ increases.
\begin{figure}[h!]
    \subfigure[The curves $C_t$ for $t=0.8$ and $n=3, \dots, 9$.]
      {\includegraphics[width=0.3 \textwidth]{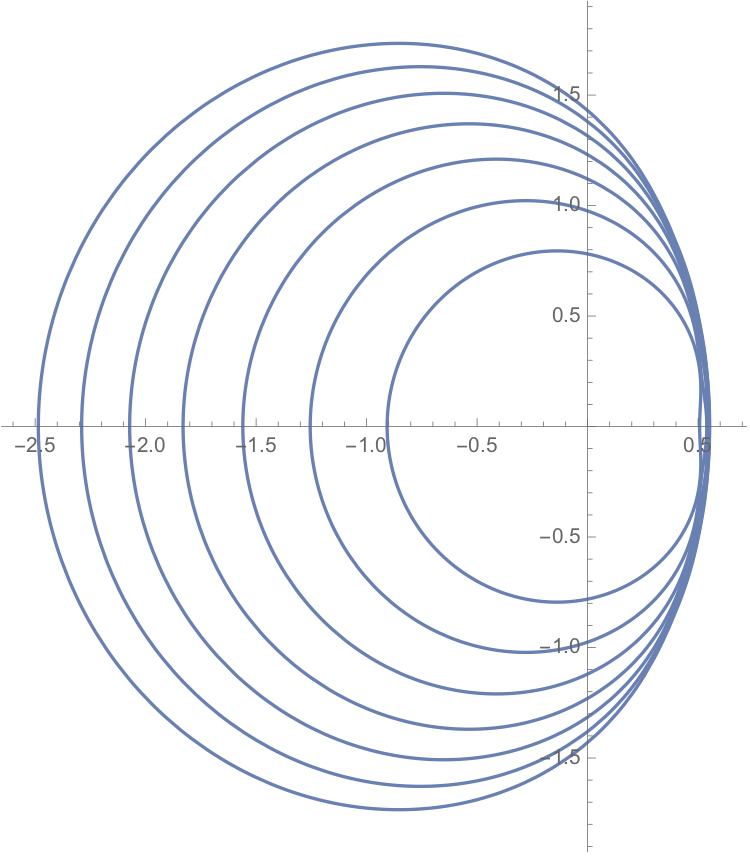}}
    \quad \quad
    \subfigure[The curves $t +(1-t^2)C_t$ for $t=0.8$ and $n=3, \dots, 9$ inside $\mathbb{D}$.]
      {\includegraphics[width=0.35 \textwidth]{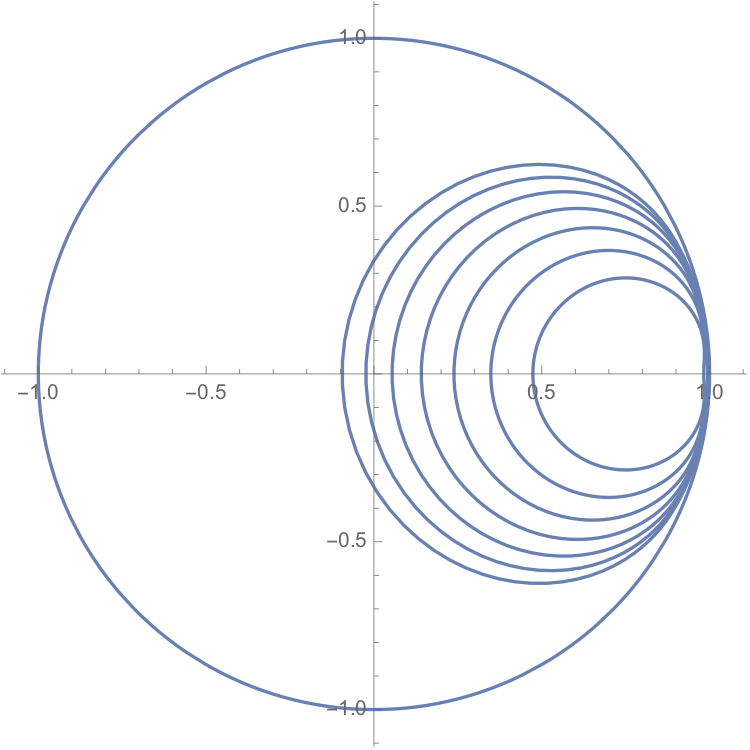}}
  \caption{\textsl{A selection of the curves $C_t$ in $W(A_t)$ and $t +(1-t^2)C_t$ in $W(S_{\Theta_t})$.}}
  \label{fig:Ct}
\end{figure}
The formula also implies that $C_t$ is a closed curve, symmetric across the $x$-axis. Setting $s=0, \pi$ gives two points in $W(A_t)$ and taking their average gives the point
 \begin{equation} \label{eqn:center} \hat{c}_t:= -\sum_{\substack{2 \le k \le n-1 \\ k \text{ even } } } t^{k-1} 
\Bigg( (n-k) \cos\left( \tfrac{k \pi}{n+1}\right) + \frac{ \sin\left( \frac{\pi(n-k)}{n+1}\right)}{\sin\left(\frac{\pi}{n+1}\right)}\Bigg),\end{equation}
which must be in $W(A_t)$ by convexity.

As mentioned earlier, when $n=3$, Crouzeix's work in \cite{C13} provides the exact boundary of $W(A_t)$.  In this $3\times 3$ case,  $C_t$ appears to closely approximate $\partial W(A_t)$, especially for small values of $t$. This phenomenon is illustrated in Figure \ref{fig:3x3}.
\begin{figure}[h!]
    \subfigure[$t =0.55.$]
      {\includegraphics[width=0.3 \textwidth]{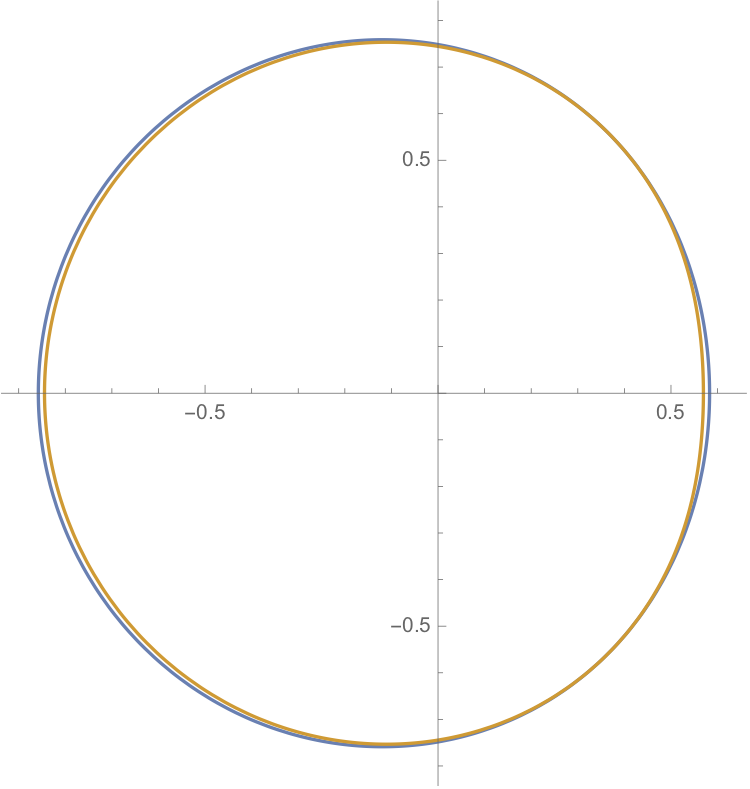}}
    \quad 
    \subfigure[$t =0.75$.]
      {\includegraphics[width=0.3 \textwidth]{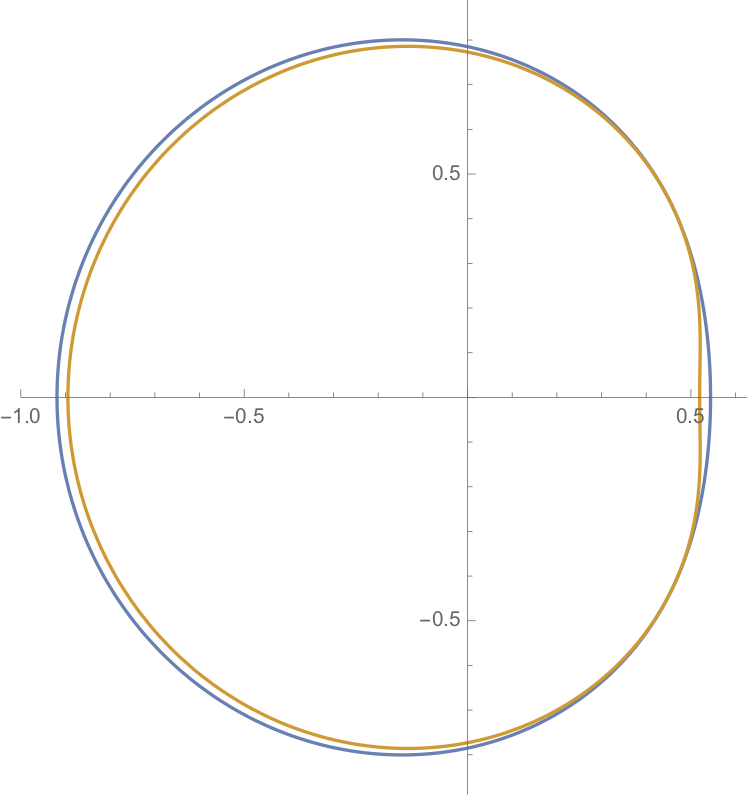}}
      \quad
          \subfigure[$t =0.95$.]
      {\includegraphics[width=0.28 \textwidth]{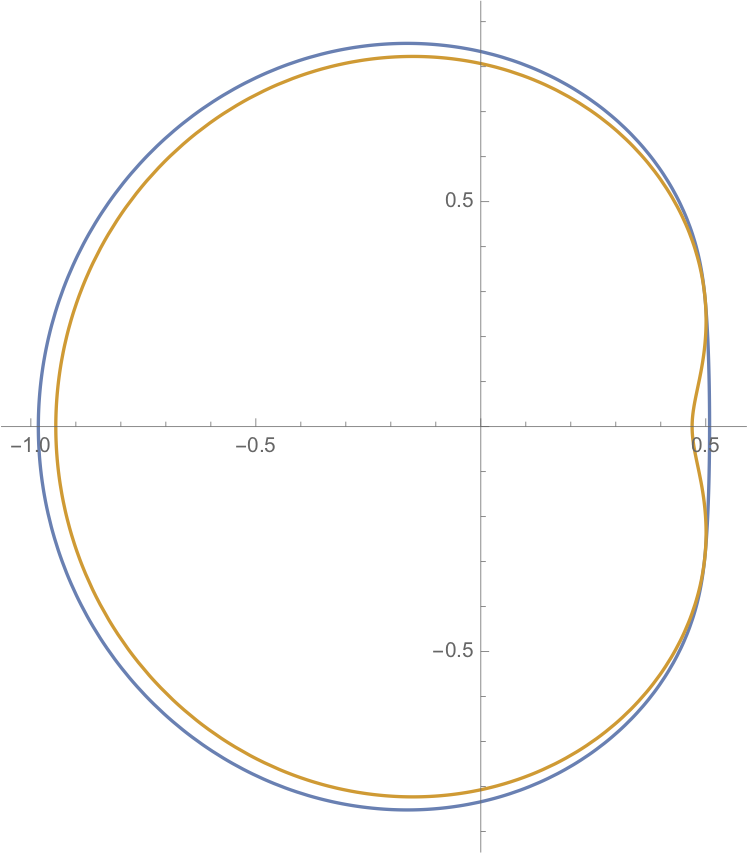}}
  \caption{\textsl{For $n=3$, a selection of the curves $C_t$ and boundaries $\partial W(A_t)$.}}
  \label{fig:3x3}
\end{figure}

\end{remark}

\subsection{Applications of $C_t$.} We now use these curves to study $W(S_{\Theta})$, for $\Theta$ unicritical. First, we can use them to identify large circles in $W(S_{\Theta})$ for small values of $n$.

\begin{theorem}\label{prop: phd} Let $\Theta$ be unicritical with  $\deg \Theta =n$. Then:
\begin{itemize}
\item[i.] If $n \ge 3$, $W(S_{\Theta})$ always contains a pseudohyperbolic disk of radius $\cos(\tfrac{\pi}{4}) =\tfrac{1}{\sqrt{2}}$.
\item[ii.]  If $n \ge 4$, $W(S_{\Theta})$ always contains a pseudohyperbolic disk of radius  $\cos(\tfrac{\pi}{5})=\frac{1}{4} (1 + \sqrt{5})$.  
\end{itemize}
\end{theorem}

\begin{proof} Without loss of generality, we can assume the unicritical $\Theta$ has its zero $t\in [0,1)$ and denote the function by $\Theta_t$. To prove (i), by the nested property of these numerical ranges, we can assume $n=3$. Then, by Remark~\ref{rem:ct}, the points on  $C_t$ are given by $f(s):=\tfrac{1}{\sqrt{2}} e^{is} - \tfrac{t}{4} e^{i(2s)}$ for $s \in [0, 2\pi)$. A simple computation gives
\begin{equation} \label{eqn:Ct3} |f(s) + \tfrac{t}{4} |^2 = | \tfrac{1}{\sqrt{2}} - \tfrac{t}{4} e^{is} + \tfrac{t}{4} e^{-is}|^2 = \tfrac{1}{2} + \tfrac{t^2}{4}\sin^2(s) \ge \tfrac{1}{2}.\end{equation}
Thus, $C_t \subseteq \mathbb{C} \setminus D_{1/\sqrt{2}} (-\tfrac{t}{4})$ and looking at $s=0, \pi, 2\pi$, the curve $C_t$ begins at $\tfrac{1}{\sqrt{2}} - \tfrac{t}{4}$, goes through $-(\tfrac{1}{\sqrt{2}} + \tfrac{t}{4})$, and ends back at  $\tfrac{1}{\sqrt{2}} - \tfrac{t}{4}$. These facts combined with the $x$-axis symmetry of $C_t$ implies that the convex hull of $C_t$ (and hence $W(A_t)$) contains $D_{1/\sqrt{2}} (-\tfrac{t}{4}).$ Thus, 
\begin{equation} \label{eqn:psd1} D_{1/\sqrt{2}(1-t^2)} (t-(1-t^2)\tfrac{t}{4}) \subseteq W(S_{\Theta_t}).\end{equation}
 This disk is also a pseudohyperbolic disk. To determine its radius $r(t)$, one can solve \eqref{eqn:r} with $c = t-(1-t^2)\tfrac{t}{4}$ and $R = \tfrac{1}{\sqrt{2}}(1-t^2)$
to conclude that
\[ r(t) = \tfrac{1}{32} \left( 24\sqrt{2} -t^2 \sqrt{2} +t^4 \sqrt{2} -\sqrt{128 -96t^2+98t^4-4t^6 +2t^8}\right).\]
 Solving $r(t) = \tfrac{1}{\sqrt{2}}$ yields only $t=0,1$ on $[0,1]$. As $r(\tfrac{1}{2}) > \tfrac{1}{\sqrt{2}}$, continuity implies that $r(t) \ge \tfrac{1}{\sqrt{2}}$ on $[0,1)$, which shows that the disk in \eqref{eqn:psd1} has pseudohyperbolic radius at least $\tfrac{1}{\sqrt{2}}$ and completes the proof of (i).
 
To prove (ii), we can assume that $n=4$. Then since $\cos(\pi/5) =\frac{1}{4} (1 + \sqrt{5})$, the formula for $C_t$ in Remark~\ref{rem:ct}  shows that the points on  $C_t$ are given by $f(s):=\cos(\tfrac{\pi}{5}) e^{is} -\tfrac{t}{\sqrt{5}} e^{i(2s)} + \tfrac{t^2}{4}(1-\tfrac{1}{\sqrt{5}}) e^{i(3s)}$ for $s \in [0, 2\pi]$.  We will examine disks in $W(A_t)$ centered at $\tfrac{-t}{\sqrt{5}}$ and thus, must analyze the quantity
 \begin{equation} \label{eqn:Ct4}  |f(s) + \tfrac{t}{\sqrt{5}}|^2 = \left |\tfrac{1}{4} (1 + \sqrt{5})  +\tfrac{t^2}{4}(1-\tfrac{1}{\sqrt{5}}) \cos(2s) \right|^2 + \left| -\tfrac{2t}{\sqrt{5}} \sin (s) +\tfrac{t^2}{4}(1-\tfrac{1}{\sqrt{5}}) \sin(2s) 
 \right|^2. 
 \end{equation}
 Setting $w = \cos(s)$ and simplifying \eqref{eqn:Ct4}, we can conclude that the right-hand side of \eqref{eqn:Ct4} is equal to 
 \[
 \begin{aligned}  \tfrac{3+\sqrt{5}}{8} + \tfrac{t^2}{2 \sqrt{5}}  
 + \tfrac{3-\sqrt{5}}{40}t^4 + \left( \tfrac{4-\sqrt{5}}{5} t^2 + \tfrac{2-2\sqrt{5}}{5} t^3 \cos(s)\right) \sin(s)^2 \\
 =\tfrac{3+\sqrt{5}}{8} + \tfrac{t^2}{2 \sqrt{5}}  
 + \tfrac{3-\sqrt{5}}{40}t^4 +\left( \tfrac{4-\sqrt{5}}{5} t^2 + \tfrac{2-2\sqrt{5}}{5} t^3 w \right)(1-w^2)\\
 =  \tfrac{3+\sqrt{5}}{8} + \tfrac{t^2}{2 \sqrt{5}} + g(t,w). 
 \end{aligned}
\]
 A straightforward, though somewhat tedious, calculus computation shows that $g(t,w) \ge 0$ on $[0,1] \times [-1,1]$. Then the same arguments used in the proof of (i) imply that the Euclidean disk with center 
 \[  c= t - (1-t^2) \tfrac{t}{\sqrt{5}} \text{ and radius } R = (1-t^2) \sqrt{ \tfrac{3+\sqrt{5}}{8} + \tfrac{t^2}{2 \sqrt{5}} }\] 
 is in  $W(S_{\Theta_t}).$ As before, we can then solve \eqref{eqn:r} to recover a formula for the associated pseudohyperbolic radius:
 \[ r(t) =\frac{g_1(t) - \sqrt{g_2(t)}}{\sqrt{g_3(t)}},\]
 where 
 \[
 \begin{aligned}
g_1(t) &=  
  25 \sqrt{2} + 55 \sqrt{10} + 75 \sqrt{2} t^2 - 23 \sqrt{10} t^2 - 20 \sqrt{2} t^4 + 8 \sqrt{10}  t^4 \\
g_2(t) &= 100\Big( 75 - 25 \sqrt{5} - 178 t^2 + 
      78 \sqrt{5} t^2 + 233.4 t^4 - 105 \sqrt{5} t^4 \\
      & \qquad \qquad - 96.8 t^6 + 
      42.4 \sqrt{5} t^6 + 14.4 t^8 - 6.4 \sqrt{5} t^8\Big)  \\
      g_3(t) &= 40^2( 15 + 5 \sqrt{5} + 4 \sqrt{5} t^2). \\
\end{aligned} 
\]
If $r(t) = \cos( \tfrac{\pi}{5})$, algebraic manipulations imply that $t$ is also a zero of 
\[  (g_1(t)^2+g_2(t) -\cos(\tfrac{\pi}{5})^2g_3(t))^2-4g_2(t) g_1(t)^2,\]
which is a degree $10$ polynomial with a factor of $t^2$, so it has a double zero at $t=0$. One can use numerical software to see that the other $8$ zeros of this polynomial lie far outside of the interval $[0,1]$.   As $r(\tfrac{1}{2}) >  \cos( \tfrac{\pi}{5})$, we can thus deduce that $r(t) \ge \cos( \tfrac{\pi}{5})$ for all $t \in (0,1)$ and so, $W(S_{\Theta_t})$ contains a disk with pseudohyperbolic radius at least $\cos(\tfrac{\pi}{5})$.  \end{proof}

 \begin{remark} \label{rem:n5} If $n=5$, one can similarly parameterize $C_t$ with $f(s)=\tfrac{\sqrt{3}}{2} e^{is} -\tfrac{7t}{12} e^{i(2s)} +\tfrac{t^2\sqrt{3}}{6} e^{i(3s)} -\tfrac{t^3}{12} e^{i(4s)}$ and recall that $\hat{c}_t :=- \tfrac{7t}{12}-\tfrac{t^3}{12}$ from \eqref{eqn:center} is in $W(A_t)$. Unfortunately,
$ | f(s) - \hat{c}_t|^2$
 does not simplify as much as in the $n=3$ and $n=4$ cases and so, we cannot proceed as in the proof of Theorem \ref{prop: phd}. Instead, we can rephrase the investigation as: ``Is the disk with pseudohyperbolic radius $\cos(\tfrac{\pi}{6})$ and Euclidean center
 \[ c(t) = t-(1-t^2)(\tfrac{7t}{12} + \tfrac{t^3}{12})\]
 inside the convex hull of $t + (1-t^2)C_t$?''  To prove this, one can use \eqref{eqn:r} to solve for the Euclidean radius $R(t)$ of that disk to get
 \[ R(t) = \tfrac{1}{12} (7 \sqrt{3} - \sqrt{(1 + 5 t^2 + t^4) (3 + 10 t^2 + 7 t^4 + t^6)}).\]
 Then to deduce the desired disk is inside $t + (1-t^2)C_t$,  one just needs to show that the Euclidean disk with center $\hat{c}_t $ and radius $R(t)/(1-t^2)$ is inside the convex hull of $C_t$. This will follow if one can establish
 \[  | f(s) - \hat{c}_t|^2 \ge \tfrac{R(t)^2}{(1-t^2)^2}.\]
This inequality can be checked in Mathematica, which indicates that for $0.01 \le t \le 0.99$ the inequality holds. It seems very likely that the inequality holds for all $t\in [0,1]$, but the Mathematica minimize command appears less stable near the endpoints $t=0,1$. This indicates that, when $n=5$, there should generally be a pseudohyperbolic disk of radius $\cos(\tfrac{\pi}{6})$ inside $W(S_{\Theta_t})$.
 \end{remark}
 
 The following corollary is an immediate application of Theorem~\ref{prop: phd} and Corollary \ref{cor:main}. The $n=3$ case also follows from results in \cite{C13}. The $n=4$ case appears to be new.
 
 \begin{corollary}  Let $\Theta, B$ be finite Blaschke products with $\deg B < \deg \Theta$. Let $\Theta$ be unicritical with $\deg \Theta$ equaling $3$ or $4$. Then
 \[  \max \{|B(z)|: z \in W(S_\Theta) \} \ge \tfrac{1}{2}.\]  \end{corollary}
 
These results motivate questions about when numerical ranges of compressions of shifts contain large pseudohyperbolic disks. These questions are explored more in Section \ref{sec:examples}. 

For now, recall that Crouzeix's conjecture states: given a square matrix $A$, the best constant $C$ for which 
 \begin{equation}
 \label{eq:cc}
\|p(A)\| \le C \max_{z \in W(A)} |p(z)| 
\end{equation} for all polynomials $p$ is $C = 2$. Using Proposition~\ref{prop:Ct}, we can study Crouzeix's conjecture for compressed shifts associated to unicritical $\Theta$ with degree $3,4,5$. We first obtain the following:
 
 \begin{proposition}\label{prop:C} Let $\Theta$ be unicritical with $\deg \Theta=3$. Then for every polynomial $p \in \mathbb{C}[z]$,
  \[\|p(S_{\Theta})\| \le \frac{ \sqrt{13 + \sqrt{41}}}{2} \max_{z \in W(S_{\Theta})}|p(z)|.\]
 \end{proposition}
 
 Before proceeding to the proof, a few comments are in order. First in \cite{C13}, Crouzeix proved that the numerical range of a $3 \times 3$ nilpotent matrix is a $2$-spectral set; that is, in this case the constant $C$ in \eqref{eq:cc} can be taken to be $2$. Because $A_t$ from \eqref{eqn:At} is nilpotent, that establishes Proposition \ref{prop:C} but with constant $2$.
 
 Our proof here is simpler but gives the weaker constant  $\frac{ \sqrt{13 + \sqrt{41}}}{2} \approx 2.20245$. However, with some reliance on Mathematica, our arguments do extend to the $n=4$ and $n=5$ cases. In those situations, there is a range of $t$-values (i.e. a range for the modulus of the zero of the unicritical $\Theta$) where the constant in \eqref{eq:cc} with $A = S_{\Theta}$ is less than $2$. For the proofs, we require the following remark.
 
 \begin{remark}\label{rem:KM} Fix  $t\in [0,1),$ recall the curve $C_t \subseteq W(A_t)$ from Proposition \ref{prop:Ct} where $A_t$ is defined in \eqref{eqn:At}, and let $g$ be a polynomial with $g(\mathbb{T}) = C_t$. Here we claim that $g(\mathbb{D})$ is also contained in $W(A_t)$. To see this, note that the boundary of $g(\mathbb{D})$ is contained in $C_t$ by the open mapping theorem. Let $K$ denote the convex hull of $\overline{g(\mathbb{D})}$. Since $K$ is compact and convex, the Krein-Milman theorem implies that $K$ is the convex hull of its extreme points. If $z$ is an extreme point of $K$, then $z$ is in the boundary of $g(\mathbb{D}) \subseteq C_t$. Therefore $g(\mathbb{D})$ is in the convex hull of $C_t$ and, by the convexity of $W(A_t)$, we have $g(\mathbb{D}) \subseteq W(A_t)$.\end{remark}
 
 \begin{proof}[Proof of Proposition~\ref{prop:C}] Without loss of generality, we can assume the unicritical $\Theta$ has its only zero $t\in [0,1)$ and will establish the inequality for $M_t$, the previously-discussed matrix representation of $S_\Theta$.  For constants $\alpha, \beta$ and a matrix $A$, we have $W(\alpha I + \beta A) = \alpha + \beta W(A)$. Therefore, the matrix $A_t$ from \eqref{eqn:At} satisfies equation \eqref{eq:cc} with constant $C$ if and only if $M_t$ satisfies the equation with the same constant. Thus, we work with the matrix $A_t$. 
 
By Remark~\ref{rem:ct}, $C_t$ is parameterized by $\tfrac{1}{\sqrt{2}} e^{is} - \tfrac{t}{4} e^{i(2s)}$, for $s \in [0, 2\pi)$. Then if we define $g$ by $g(z) = \frac{1}{\sqrt{2}} z - \frac{t}{4} z^2$, we have $g(\mathbb{T}) = C_t$.  By Remark \ref{rem:KM}, we immediately have $g(\mathbb{D}) \subseteq W(A_t)$. 
 Let $J_3$ denote the  $3 \times 3$ Jordan block with zeros on the diagonal and define the matrices
  \begin{equation} \label{eqn:Xt3} B_t =\begin{pmatrix}
 0 & \sqrt{2} & -\tfrac{t}{\sqrt{2}}\\[.3 em]
 0 & 0 & \sqrt{2}\\
 0 & 0 & 0
 \end{pmatrix} ~\mbox{and}~ X_t = \begin{pmatrix}
1& 0 & 0\\
 0 & \tfrac{1}{\sqrt{2}} & \tfrac{t}{4}\\[.2 em]
 0 & 0 & \tfrac{1}{2}
 \end{pmatrix}.\end{equation}
 Then it can be checked that  $g(B_t) = A_t$ and  $B_t = X_t \cdot J_3 \cdot X_t^{-1}$, and for $p \in \mathbb{C}[z]$ we have 
 \begin{align}
& \|p(A_t)\| =  \|(p \circ g)(B_t)\| = \|X_t (p \circ g)(J_3) X_t^{-1}\| &\notag\\
 & \le   \|X_t\| \cdot \|X_t^{-1}\| \cdot \|(p \circ g)(J_3)\| \le \|X_t\| \cdot \|X_t^{-1}\| \sup_{z \in \mathbb{D}} |(p \circ g)(z)|,
 \end{align} where we used von Neumann's inequality in the last line. 
We know that $z \in \mathbb{D}$ implies that $g(z) \in W(A_t)$. Thus, 
 \[\|p(A_t)\| \le  \|X_t\| \cdot \|X_t^{-1}\| \sup_{z \in W(A_t)} |p(z)|.\]
 A Mathematica computation shows that 
 \[ \|X_t\| \cdot \|X_t^{-1}\| = \tfrac{1}{2} \sqrt{12 + t^2 + \sqrt{16 + 24 t^2 + t^4}},\]
 which  is increasing in $t$ and satisfies
 \[2 \le \|X_t\| \cdot \|X_t^{-1}\| \le  \frac{ \sqrt{13 + \sqrt{41}}}{2},\]
 which completes the proof.
 \end{proof}

  \begin{remark} Let $\Theta_t$ denote a unicritical Blascke product with zero at  $t \in [0,1)$ and degree $n$. By Remark \ref{rem:ct}, we have tractable formulas for $C_t$ when $n=4$ and $n=5$. In these cases, we can adapt the arguments from Proposition
  \ref{prop:C} to show that $W(S_{\Theta_t})$ is a $\|X_t\| \cdot \|X_t^{-1}\|$ spectral set for $S_{\Theta_t}$, for a (soon-to-be specified) matrix $X_t$. As before, we work with $A_t$.
  
  First, for the $n=4$ case, set $g(z) = \tfrac{1}{4}(1 + \sqrt5) z -\tfrac{t}{\sqrt{5}} z^2 + \tfrac{t^2}{4}\left(1-\tfrac{1}{\sqrt{5}}\right)z^3$ and define
  \[ B_t = \begin{pmatrix} 0 & -1 + \sqrt{5} & (9 - \tfrac{21}{\sqrt{5}}) t & 
    \tfrac{1}{5} (-537 + 241 \sqrt{5}) t^2 \\[.2 em]
    0 & 0 & -1 + \sqrt{5} & (9 - \tfrac{21}{\sqrt{5}}) t \\[.2 em]
     0 & 0 & 0 & -1 + \sqrt{5}\\
     0 & 0 &0 &0 \end{pmatrix}  \]
  and 
  \begin{equation} \label{eqn:Xt4} X_t = \begin{pmatrix} 1 & 0 & 0 & 0 \\
   0 & \tfrac{1}{4} (1 + \sqrt{5}) & -\tfrac{3}{40} (-5 + \sqrt{5}) t & -\tfrac{t^2}{
  8 \sqrt{5}} \\[.2 em]
    0 & 0 & \tfrac{1}{8} (3 + \sqrt{5}) & \tfrac{3 t}{4 \sqrt{5}} \\[.2 em]
     0 & 0 & 0 & \tfrac{1}{8} (2 + \sqrt{5}) \end{pmatrix}\end{equation}
  Then, $g(B_t ) = A_t$,   $B_t = X_t \cdot J_4 \cdot X_t^{-1}$, and 
  the same arguments as in the proof of Proposition \ref{prop:C} imply that for all polynomials $p$, 
   \[\|p(A_t)\| \le  \|X_t\| \cdot \|X_t^{-1}\| \sup_{z \in W(A_t)} |p(z)|.\]
However, for $n=4$, there is not a nice formula for $\|X_t\| \cdot \|X_t^{-1}\|$. Still, the Mathematica maximization tool reveals that 
  for $t \in (0,0.9999)$,  
  \[ \|X_t\| \cdot \|X_t^{-1}\| \le 2.38\]
  and indeed, if $ t \in (0, .42)$, $\|X_t\| \cdot \|X_t^{-1}\| <2$. Thus, this argument shows that for $n=4$, Crouzeix's conjecture holds for all $S_{\Theta_t}$ with $t\in (0, 0.42).$

  Basically, the same argument works if $n=5.$ In this case, $g(z) = \tfrac{\sqrt{3}}{2}z -\tfrac{7t}{12} z^2 +\tfrac{t^2\sqrt{3}}{6} z^3 -\tfrac{t^3}{12} z^4,$
\[ B_t = \begin{pmatrix} 0 & \tfrac{2}{\sqrt{3}} & -\tfrac{4 t}{9 \sqrt{3}} & \tfrac{34 t^2}{
   81 \sqrt{3}} & -\tfrac{278 t^3}{729 \sqrt{3}} \\[.3 em]
    0 & 0 & \tfrac{2}{\sqrt{3}} & -\tfrac{4 t}{9 \sqrt{3}} & \tfrac{34 t^2}{
   81 \sqrt{3}} \\[.2 em]
     0 & 0 & 0 &  \tfrac{2}{\sqrt{3}} & -\tfrac{4 t}{9 \sqrt{3}} \\[.2 em]
     0 & 0 & 0 & 0 &  \tfrac{2}{\sqrt{3}}  \\[.2 em]
     0 &  0 &  0 & 0 & 0 \end{pmatrix} \]
and 
\begin{equation} \label{eqn:Xt5} X_t = \begin{pmatrix} 1 & 0 & 0 &0 & 0 \\
 0 & \tfrac{\sqrt{3}}{2} & \tfrac{t}{6} & -\tfrac{t^2}{8 \sqrt{3}} &\tfrac{ t^3}{
  144} \\[.2 em]
  0 & 0 & \tfrac{3}{4} & \tfrac{t}{2 \sqrt{3}} & -\tfrac{7 t^2}{72} \\[.4 em]
  0 & 0 & 0 & \tfrac{3 \sqrt{3}}{8} & \tfrac{3 t}{8} \\[.1 em] 
  0 & 0 & 0 & 0 & \tfrac{9}{16}\end{pmatrix}.\end{equation}
Then, $g(B_t ) = A_t$,   $B_t = X_t \cdot J_5 \cdot X_t^{-1}$, and 
  the same arguments imply that for all polynomials $p$, 
   \[\|p(A_t)\| \le  \|X_t\| \cdot \|X_t^{-1}\| \sup_{z \in W(A_t)} |p(z)|.\]
As in the $n=4$ case, for $n=5$, there is not a nice formula for $\|X_t\| \cdot \|X_t^{-1}\|$. Still, the Mathematica maximization tool reveals that 
  for $t \in (0.0001,1)$,  
  \[ \|X_t\| \cdot \|X_t^{-1}\| \le 2.51\]
  and indeed, if $ t \in (0.0001, 0.5)$, $\|X_t\| \cdot \|X_t^{-1}\| <2$. Here, we consider $t$ values away from $0$ because the maximization tool seems to 
  be somewhat unstable near $t=0$. Regardless, this argument shows that, for $n=5$, Crouzeix's conjecture holds for all $S_{\Theta_t}$ with $t\in (.0001, 0.5).$
  
For $n \ge 6$, the methods we used to compute $B_t$ and $X_t$ are no longer manageable. Still, we conjecture that a similar argument should work in theory, if not in practice, for these higher values of $n$.
  \end{remark}

\section{Pseudohyperbolic Disks and Numerical Ranges} \label{sec:examples}

Several of our results, particularly Corollary \ref{cor:main} and Theorem \ref{prop: phd}, require or show the existence of large pseudohyperbolic disks contained inside the numerical ranges $W(S_\Theta)$. These results lead naturally to the following question:

\medskip
\noindent
If $\Theta$ is a finite Blaschke product with $\deg \Theta =n$, does $W(S_\Theta)$ necessarily contain a pseudohyperbolic disk with pseudohyperbolic radius $(\frac{1}{2})^{1/(n-1)}$?
\medskip

However, the answer to this question is no! To illustrate this, in the following example, we provide a family of degree-$2$ Blaschke products $\Theta$ such that $W(S_\Theta)$ does not contain any pseudohyperbolic disk of the form $D_\rho(z_0, \tfrac{1}{2}).$  However, since $\dim K_\Theta=2$, the $2\times 2$ result in \cite{Crouzeix1} still implies that $S_\Theta$ satisfies the inequality in Crouzeix's conjecture.

\begin{example}  Set $\Theta(z) = \frac{z^2-t^2}{1-t^2 z^2}$ for $t \in (0,1)$. Then one matrix representation of $S_\Theta$ is
\[ M_\Theta :=  \begin{bmatrix}  t & 1-t^2 \\ 0 & -t \end{bmatrix} \]
and the elliptical range theorem implies that $W(S_\Theta)$ is the elliptical disk with foci $\pm t$ and minor axis $1-t^2$. Equivalently, $W(S_\Theta)$ is exactly the set of points $z = x +i y$ satisfying
\begin{equation} \label{eqn:ellipse} \frac{ 4 x^2}{(1+t^2)^2} + \frac{ 4 y^2 }{(1-t^2)^2} \le 1.\end{equation}
Assume that some  $D_\rho(z_0, \tfrac{1}{2}) \subseteq W(S_\Theta)$ for $z_0 = x_0 + i y_0$. We show this leads to a contradiction for $t >
\sqrt{3/4}$. First, note that $z_0 \in W(S_\Theta)$ and so \eqref{eqn:ellipse} implies that 
\begin{equation} \label{eqn:tz0} x_0^2 + y_0^2 \le \tfrac{1}{4}(1+t^2)^2 \text{ and so, } 1 - |z_0|^2 \ge  1-  \tfrac{1}{4}(1+t^2)^2.\end{equation}
Recall from \eqref{eqn:phd1} that $D_\rho(z_0, \tfrac{1}{2})$ is also a Euclidean disk with center $c$ and radius $R$ defined by 
\[ c  = \frac{ \tfrac{3}{4} z_0}{1-  \tfrac{1}{4}|z_0|^2} \ \text{ and } R = \frac{ \tfrac{1}{2}(1-|z_0|^2)}{1-  \tfrac{1}{4} |z_0|^2}.\] 
 By the assumption that $D_\rho(z_0, \tfrac{1}{2}) \subseteq W(S_\Theta)$, we see that $ c \pm i R$ must satisfy \eqref{eqn:ellipse}.  So, with $|y| = R$ in \eqref{eqn:ellipse} we immediately obtain
\[ \frac{4 R^2}{(1-t^2)^2} \le 1.\]
It must be the case that either $|z_0| <  \sqrt{3/4}$ or $|z_0| \ge  \sqrt{3/4}$. If $|z_0| < \sqrt{3/4}$, then 
\[  \frac{(1- \frac{3}{4})^2}{(1-t^2)^2} \le \frac{(1-|z_0|^2)^2}{(1-  \tfrac{1}{4}|z_0|^2)^2 (1-t^2)^2} = \frac{4 R^2}{(1-t^2)^2} \le 1,\] 
which only holds if $t^2 \le  \frac{3}{4}$, or $t \le \sqrt{3/4}.$ Similarly, if $|z_0| \ge  \sqrt{3/4}$, \eqref{eqn:tz0} implies 
\begin{equation} \label{eqn:tz1} \left(\frac{16}{13}\right)^2  \frac{(1- \tfrac{1}{4}(1+t^2)^2)^2}{ (1-t^2)^2} \le \frac{(1-|z_0|^2)^2}{(1-  \tfrac{1}{4} |z_0|^2)^2 (1-t^2)^2} \le 1.\end{equation}
A computation shows that 
\[ \lim_{t\rightarrow 1} \frac{1- \tfrac{1}{4}(1+t^2)^2}{1-t^2}  =1,\]
and so for $t$ sufficiently close to $1$, \eqref{eqn:tz1} has to fail. More specifically, one can check that \eqref{eqn:tz1} only holds if $t \le  \tfrac{1}{2}.$ Combining our two computations implies that if $t > \sqrt{3/4}$, both inequalities fail and then $W(S_\Theta)$  cannot contain a pseudohyperbolic disk $D_\rho(z_0, 1/2)$. 
\end{example}

In contrast, Theorem \ref{prop: phd} shows that if $\Theta$ is unicritical and $\deg \Theta =3$ or $\deg \Theta =4$, then it includes a pseudohyperbolic disk of the radius $(\frac{1}{2})^{1/2}$ or $(\frac{1}{2})^{1/3}$ respectively. The following example shows that (unsurprisingly) this result also holds for unicritical $\Theta$ with $\deg \Theta =2$.

\begin{example} Set $\Theta_t(z) =\left( \frac{z-t}{1-t z}\right)^2$ for $t \in [0,1)$, so that $\Theta_t$ is a degree $2$ unicritical Blaschke product with its zero at $t$. We will show that there is a pseudohyperbolic disk $D_\rho(z_0, \tfrac{1}{2})$ contained in the numerical range $W(S_{\Theta_t})$.
One matrix representation of $S_{\Theta_t}$ is
\[ M_{\Theta_t} =  \begin{bmatrix}  t & 1-t^2 \\ 0 & t \end{bmatrix}. \]
Then the elliptical range theorem implies that $W(S_{\Theta_t})$ is the Euclidean disk whose center $c(t) = t$ and radius $R(t) = \tfrac{1}{2}(1-t^2)$. This Euclidean disk is also a pseudohyperbolic disk $D_\rho(z_0(t), r(t))$ with center $z_0(t) \in \mathbb{R}^+$ and radius $r(t)$ that must satisfy the equations  \eqref{eqn:z0} and \eqref{eqn:r}. Solving those equations gives $r(0) = \tfrac{1}{2}, z_0(0) =0$, and for $t \ne 0$, 
\[ 
\begin{aligned}
r(t) &=\tfrac{1}{4}\left( 5 - t^2 -\sqrt{(1-t^2)(9-t^2)}\right) \\
z_0(t) & = \tfrac{1}{8t}\left(3 + 6 t^2 - t^4 -(1-t^2)\sqrt{(1-t^2)(9-t^2)}\right).
\end{aligned}\]
A calculus computation implies that $r(t)$ is increasing in $t$ on $[0,1)$ and $r(0) = \tfrac{1}{2}$. Thus, each $W(S_{\Theta_t})$ equals some $D_\rho(z_0(t), r(t))$ with $r(t) \ge \tfrac{1}{2}$, which gives the desired result. Moreover, if $t \ne 0$, then $r(t) > \tfrac{1}{2}$ and so, we can perturb the zeros slightly from $t$ to some $t_1, t_2$ and the resulting $\Theta$ will still include some  $D_\rho(z_0, \tfrac{1}{2})$ in its associated numerical range $W(S_\Theta)$.
\end{example}

That example combined with Theorem \ref{prop: phd} motivates the following open question:

\begin{question} If $\Theta$ is unicritical with $\deg \Theta =n$,  does $W(S_\Theta)$ necessarily contain a pseudohyperbolic disk with pseudohyperbolic radius $(\frac{1}{2})^{1/(n-1)}$?\end{question}

We conjecture that the answer is yes. It is worth noting that these large pseudohyperbolic disks typically cannot be centered at the zero of the unicritical Blaschke product. It is easiest to see this by examining the degree-$2$ situation, as follows.

\begin{example} Assume $\Theta_t$ is unicritical with its zero at $t \in [0,1)$. Then $W(S_{\Theta_t})$ is the closed Euclidean disk with center $c_1= t$ and radius $R_1=\tfrac{1}{2}(1-t^2)$. Meanwhile using \eqref{eqn:phd1},  $D_\rho (t, \tfrac{1}{2} )$ is the Euclidean disk with center $c_2$ and radius $R_2$ given by
\[ c_2 = \frac{ \tfrac{3}{4}t}{1-  \tfrac{1}{4}t^2} \ \text{ and } R_2 = \frac{ \tfrac{1}{2}(1-t^2)}{1-  \tfrac{1}{4} t^2}.\] 
The boundary circles of two such Euclidean disks intersect in exactly two points (and hence, neither disk contains the other) if and only if 
\begin{equation} \label{eqn:circles} (R_1-R_2)^2 < |c_1-c_2|^2 < (R_1+R_2)^2.\end{equation}
Computing those quantities directly gives 
\[
R_1 + R_2 = \frac{(1-t^2)(8-t^2)}{2(4-t^2)}, \ \ \ R_2-R_1 = \frac{t^2(1-t^2)}{2(4-t^2)}, \   \ \ c_1 - c_2 = \frac{t(1-t^2)}{4-t^2},\]
and comparing them shows that \eqref{eqn:circles} 
holds as long as $t \ne 0.$ This shows that if $n=2$, $D_\rho\left(t, \tfrac{1}{2} \right)  \not \subseteq W(S_{\Theta_t})$ and  similarly, $ W(S_{\Theta_t}) \not \subseteq D_\rho\left(t,  \tfrac{1}{2}  \right)$. 
\end{example}

\section{Proof of Theorem \ref{HRrevised}}  \label{sec:RH} 
For completeness, we recall the original result of Horwitz and Rubel:

\begin{theorem}[\cite{HorwitzRubel86}] Let $A$ and $B$ be two monic Blaschke products of degree $n$. Suppose that there are $n$ distinct points $\lambda_1, \ldots, \lambda_n$ in $\mathbb{D}$ such that $A(\lambda_j) = B(\lambda_j)$ for $j = 1, \ldots, n$. Then $A = B$. \end{theorem}

 The proof  given in \cite{HorwitzRubel86}  relies on the following lemma stated under the assumptions above. However, the lemma does not use the assumption that the points are distinct. Still, because this lemma is essential to the proof of Theorem~\ref{HRrevised}, we give a detailed proof below. Then we establish Theorem \ref{HRrevised}, which handles the case where $A$ and $B$ agree at $n$ (not necessarily distinct) points in $\mathbb{D}$, when those points are counted according to multiplicity.

\begin{lemma}[\cite{HorwitzRubel86}]\label{lem:onemore}
Let $A$ and $B$ be monic Blaschke products of degree $n$. Then there exists $\lambda \in \mathbb{T}$ such that $A(\lambda) = B(\lambda)$.
\end{lemma}

\begin{proof} Suppose that $A$ has zeros $a_1, \ldots, a_n$ and $B$ has zeros $b_1, \ldots, b_n$. Note that for $\lambda \in \mathbb{T}$, we have $A(\lambda) = B(\lambda)$ if and only if $A(\lambda)/B(\lambda) = 1$ and this happens if and only if 
\[\prod_{j = 1}^n \left(\left(\frac{\lambda-a_j}{\lambda - b_j}\right)\big/\left(\frac{1 - \overline{a_j}\lambda}{1-\overline{b_j}\lambda}\right)\right) = 1.\] Since $\lambda \in \mathbb{T}$, this happens if and only if
\begin{equation}\label{eqn:1}
\prod_{j = 1}^n \left(\left(\frac{\lambda-a_j}{\lambda - b_j}\right)\big/\overline{\left(\frac{\lambda-a_j}{\lambda - b_j}\right)}\right) = 1. 
\end{equation}
Let $F(z):=\prod_{j=1}^n \frac{z-a_j}{z - b_j}.$ Then establishing \eqref{eqn:1} is equivalent to showing that $1 = F(\lambda)/\overline{F(\lambda)}$. Now define
$G(z):=F(1/z) = \prod_{j=1}^n \left(\frac{1/z-a_j}{1/z - b_j}\right)$ and note that $G$ has a holomorphic extension (also denoted by $G$) to a domain that includes $z = 0$, namely,
\[G(z) = \prod_{j = 1}^n \frac{1 - a_j z}{1 - b_j z}.\] Now, $a_j, b_j \in \mathbb{D}$ for all $j$, so there exists $\delta > 0$ such that $G$ is holomorphic and zero free on $|z| < 1 + \delta$. By \cite[Corollary 1.1.3]{Ransford95}, there exists a holomorphic function $H$ on $|z| < 1 + \delta$ such that 
\begin{equation}\label{eqn:2}
G = e^H.
\end{equation}
By definition, $G(0) = 1$, so $\IM H(0) = 2 \pi m$ for some $m \in \mathbb{Z}$. Subtracting $2 \pi m i$ from $H$ will not change \eqref{eqn:2} or the holomorphic nature of $H - 2 \pi m i$, so we may assume that $\IM H(0) = 0$. 
Since $\IM H$ is harmonic, the mean value theorem implies that
\[
0 = \IM H(0) = \frac{1}{2\pi} \int_{0}^{2\pi} \IM H(e^{i\theta}) d\theta.\]
Because $\IM H$ is continuous on $\mathbb{T}$, this implies that there must exist $\theta_0\in[0, 2\pi]$ with $\IM H(e^{i \theta_0}) = 0$. Let $\lambda:=e^{-i\theta_0}$. Then
\[F(\lambda) = G(e^{i \theta_0}) = e^{\RE [H(e^{i\theta_0})]} \in \mathbb{R}\setminus \{0\}.\]
Therefore, $F(\lambda)/\overline{F(\lambda)} = 1$, as needed.
\end{proof}

We can now prove Theorem \ref{HRrevised}. This proof uses a somewhat different argument than the proof in \cite{HorwitzRubel86}.

\begin{proof} By Lemma~\ref{lem:onemore}, there is a point $\lambda \in \mathbb{T}$ where $A(\lambda)=B(\lambda).$  Let $a_1, \dots, a_n$ be the zeros of $A$ counted according to multiplicity and define polynomials 
\[ p_a(z) = \prod_{j=1}^n (z-a_j)  \text{ \ \ and \ \ } q_a = \prod_{j=1}^n (1-\bar{a}_jz),\] 
so $A = p_a/q_a$. Define $p_b, q_b$ in an analogous way for $B$. Consider the polynomial $Q:=p_a q_b - p_b q_a$, which is the numerator for $A-B$ and observe that $\deg Q \le 2n.$ Moreover, a simple computation shows
\begin{equation} \label{eqn:Q}  z^{2n} \overline{Q(1/\bar{z})} =-Q(z).\end{equation}
Assume that $A$ and $B$ agree with multiplicity $k$ at $c \in \mathbb{D}$ with $c \ne 0$. Then $(z-c)^k$ divides $Q$ and so, \eqref{eqn:Q} implies that $(1-\bar{c} z)^k$ divides $Q$. Thus $Q$ has a zero of multiplicity $k$ at both $c \in \mathbb{D}$ and $1/\bar{c} \in \mathbb{C} \setminus \overline{\mathbb{D}}.$ 

The rest of the proof requires two cases. For the first case, assume each $\lambda_i \ne 0$. By our above arguments, again counting according to multiplicity,  $Q$ vanishes at $2n+1$ points in $\mathbb{C}$ and so, is identically $0$. Thus $A=B$.
For the second case, assume without loss of generality that $\lambda_1=0$ and $A, B$ agree with multiplicity $k$ at $\lambda_1$. Then $Q = z^k R$ for some polynomial $R$ and \eqref{eqn:Q} becomes
\[ z^{2n-k}\overline{R(1/\bar{z})}  =-z^kR(z).\]
This implies $\deg R \le 2n-2k$ and thus, $\deg Q \le 2n-k$. By the above arguments, $Q$ must vanish at $2n+1-k$ points in $\mathbb{C}$  and so is identically $0$. Thus $A=B$.
 \end{proof}

As pointed out in \cite{HorwitzRubel86}, the assumption that the Blaschke products are monic is essential; if  $A(z) = \frac{z - i/2}{1 - (i/2)z}$ and $B(z) = i \frac{z-1/2}{1-(1/2)z},$ then $A(0)=B(0)$ but clearly $A \ne B$.


\begin{thebibliography}{12}

\bibitem{baranov} A. D.  Baranov, Weighted Bernstein inequalities and embedding theorems for model subspaces. (Russian) Algebra i Analiz 15 (2003), no. 5, 138–168; translation in St. Petersburg Math. J. 15 (2004), no. 5, 733–752.

\bibitem{BCD}
C. Badea, M. Crouzeix, B. Delyon,
Convex domains and K-spectral sets.
Math. Z. 252 (2006), no. 2, 345--365.

\bibitem{Badea} C. Badea, B. Beckermann, Spectral sets, 2013, preprint.

\bibitem{B84}
R. Berman, The level sets of the moduli of functions of bounded characteristic.
{Trans. Amer. Math. Soc.} 281 (1984), no. 2, 725–744.


\bibitem{BGGRSW} K. Bickel, P. Gorkin; A. Greenbaum, T. Ransford, F. L. Schwenninger, E. Wegert, Crouzeix's conjecture and related problems. Comput. Methods Funct. Theory 20 (2020), no. 3-4, 701–728.

\bibitem{CGL18}
T. Caldwell, A. Greenbaum, K. Li,
Some extensions of the Crouzeix--Palencia result.
{SIAM J. Matrix Anal. Appl.} 39 (2018),  769--780.



\bibitem{CH13}
D. Choi, A proof of Crouzeix's conjecture for a class of matrices, {Linear Alg. Appl.} 438 (2013), 3247--3257.


\bibitem{CG}
D. Choi, A. Greenbaum,  Roots of matrices in the study of GMRES convergence and Crouzeix's conjecture. {SIAM J. Matrix Anal. Appl.} 36 (2015), no. 1, 289--301



\bibitem{CM} J. Cima and R. Mortini, One-component inner functions. Complex Anal. Synerg. 3 (2017), no. 1, Paper No. 2, 15 pp.

\bibitem{CM2} J. Cima and R. Mortini, One-component inner functions II, Advancements in complex analysis (eds D. Breaz and M. Rassias; Springer, Berlin, 2020) 39–49.

\bibitem{Cohn} B. Cohn, Carleson measures for functions orthogonal to invariant subspaces. Pacific J. Math. 103 (1982), no. 2, 347–364.

\bibitem{Crouzeix1} M. Crouzeix, Bounds for analytical functions of matrices. Integral Equations Operator Theory 48 (2004), no. 4, 461–477.

\bibitem{Cr07}
M. Crouzeix,
Numerical range and functional calculus in Hilbert space. J. Funct. Anal. 244 (2007), no. 2, 668--690.


\bibitem{C13}
 M. Crouzeix,
Spectral sets and $3\times3$ nilpotent matrices. Topics in functional and harmonic analysis, 27--42,
Theta Ser. Adv. Math., 14, Theta, Bucharest, 2013.


\bibitem{CP17}
M. Crouzeix, C. Palencia, The numerical range is a $(1 + \sqrt{2})$ spectral set, SIAM J. Matrix Anal. Appl., 38 (2017), 649--655.

\bibitem{DGM} U. Daepp, P. Gorkin, R. Mortini, Ellipses and finite Blaschke products. Amer. Math. Monthly 109 (2002), no. 9, 785–795.

\bibitem{DGSV}
U. Daepp, P. Gorkin, A. Shaffer, K. Voss, \emph{Finding ellipses. What Blaschke products, Poncelet's theorem, and the numerical range know about each other}. Carus Mathematical Monographs, 34. MAA Press, Providence, RI, 2018.


\bibitem{Eben2011} 
P. Ebenfelt, D. Khavinson, H.S. Shapiro.
Two-dimensional shapes and lemniscates. Complex analysis and dynamical systems IV. Part 1, 45–59,
Contemp. Math., 553, Israel Math. Conf. Proc., Amer. Math. Soc., Providence, RI, 2011.

\bibitem{Fuss} Fuss, N. Nova Acta Petropol. 10, 1792.

\bibitem{Fujimura} M. Fujimura,  Inscribed ellipses and Blaschke products. Comput. Methods Funct. Theory 13 (2013), no. 4, 557--573. 

\bibitem{Gaaya10} H. Gaaya, On the numerical radius of the truncated adjoint shift. Extracta Math. 25 (2010), no. 2, 165--182.


\bibitem{Gaaya12}  H. Gaaya,  A sharpened Schwarz-Pick operatorial inequality for nilpotent operators. Indiana Univ. Math. J. 61 (2012), no. 1, 223--248.



\bibitem{GarciaRoss} S. R. Garcia, W. T. Ross, A non-linear extremal problem on the Hardy space. Comput. Methods Funct. Theory 9 (2009), no. 2, 485–524.

\bibitem{Garnett} J. B. Garnett, Bounded analytic functions. Pure and Applied Mathematics, 96. Academic Press, Inc. [Harcourt Brace Jovanovich, Publishers], New York-London, 1981.

\bibitem{GauWu} H.-L. Gau, P. Y. Wu, Numerical range of $S(\phi)$, Linear and Multilinear Algebra, 45 (1998), no. 1, 49--73.

\bibitem{GauWuCond} H.-L. Gau, P. Y. Wu,
Condition for the numerical range to contain an elliptic disc. 
Linear Algebra Appl. 364 (2003), 213–222.



\bibitem{GauWu13}
H.-L. Gau, P. Y. Wu,  Numerical ranges of KMS matrices. Acta Sci. Math. (Szeged) 79 (2013), no. 3-4, 583–610.

\bibitem{GauWu14} H.-L. Gau, P. Y. Wu,  Yuan Zero-dilation indices of KMS matrices. Ann. Funct. Anal. 5 (2014), no. 1, 30–35. 

\bibitem{GKL}
C. Glader, M. Kurula, M. Lindstr\"{o}m,
Crouzeix's conjecture holds for tridiagonal $3 \times 3$ matrices with elliptic numerical range centered at an eigenvalue.
{SIAM J. Matrix Anal. Appl.} 39 (2018), no. 1, 346--364.

\bibitem{GorkinPartington} P. Gorkin, J. R. Partington, Norms of truncated Toeplitz operators and numerical radii of restricted shifts. Comput. Methods Funct. Theory 19 (2019), no. 3, 487–508.

\bibitem{GorkinWagner} P. Gorkin, N. Wagner, Ellipses and compositions of finite Blaschke products. J. Math. Anal. Appl. 445 (2017), no. 2, 1354–1366.

\bibitem{GreenbaumChoi} A. Greenbaum, D. Choi, Crouzeix's conjecture and perturbed Jordan blocks. Linear Algebra Appl. 436 (2012), no. 7, 2342–2352.

\bibitem{GO}
A. Greenbaum, M. L. Overton, Numerical investigation of Crouzeix's conjecture. Linear Algebra Appl. 542 (2018), 225--245.

\bibitem{HaagerupdelaHarpe} U. Haagerup, P.  de la Harpe, The numerical radius of a nilpotent operator on a Hilbert space.
Proc. Amer. Math. Soc. 115 (1992), no. 2, 371–379.

\bibitem{Hess} A. Hess, Bicentric quadrilaterals through inversion, Forum Geometricorum, Volume 13 (2013) 11--15.

\bibitem{Hoffman} K. Hoffman, Bounded analytic functions and Gleason parts. Ann. of Math. (2) 86 (1967), 74–111.

\bibitem{HorwitzRubel86} A. Horwitz, L. Rubel, A uniqueness theorem for monic Blaschke products. Proc. Amer. Math. Soc. 96 (1986), no. 1, 180–182.

\bibitem{knapp} M. Knapp, Sines and Cosines of Angles in Arithmetic Progression. 
Mathematics Magazine.  82 (2009) no. 5, 371-372.

\bibitem{M98} B. Mirman,  Numerical ranges and Poncelet curves. Linear Algebra Appl. 281 (1998), no. 1-3, 59–85.


\bibitem{MortiniRupp2014} R. Mortini, R. Rupp,  The symmetric versions of Rouché's theorem via $\overline{\partial}$-calculus. J. Complex Anal. 2014, Art. ID 260953, 9 pp.

\bibitem{MortiniRupp2021} R. Mortini, R. Rupp, Extension Problems and Stable Ranks: A Space Odyssey. Berkh\"auser, 2021.


\bibitem{Nicolau}
A. Nicolau, A.  Reijonen,  A characterization of one-component inner functions. Bull. Lond. Math. Soc. 53 (2021), no. 1, 42–52.


\bibitem{N93}
N. Steinmetz, The formula of Riemann-Hurwitz and iteration of rational functions, Complex Variables Theory Appl. 22 (1993), no. 3-4, 203--206.


\bibitem{Ransford95} T. Ransford,  Potential theory in the complex plane. London Mathematical Society Student Texts, 28. Cambridge University Press, Cambridge, 1995.

\bibitem{Sar67}
D. Sarason, Generalized interpolation in $H^\infty$. {Trans. Amer. Math. Soc.} 127 (1967), 179--203.



\bibitem{Stephenson} K. Stephenson. Analytic functions sharing level curves and tracts. Ann. of
Math., 123, (1986), 107-144.

\bibitem{Stephenson2} K. Stephenson, C. Sundberg, Level curves of inner functions. Proc. London
Math. Soc., 51, (1985), 77-94.



\bibitem{SFBK10}
B. Sz.-Nagy, C. Foias, H. Bercovici, L.  K\'erchy, Harmonic analysis of operators on Hilbert space, second ed., Universitext, Springer, New York, 2010.

\end{thebibliography}
\end{document}